\documentclass[a4paper, 10pt]{amsart}

\usepackage[square, numbers, comma]{natbib} 

\usepackage[usenames,dvipsnames,svgnames,table]{xcolor}
\usepackage{amsmath,amsthm,amssymb,amssymb,esint,verbatim,tabularx,graphicx}
\usepackage{fancyhdr}
\usepackage{enumerate}
\usepackage{amsmath}
\usepackage{fancyhdr}
\usepackage{epic}
\usepackage{pgf,tikz}
\usetikzlibrary{arrows}

\usepackage{color}
\usepackage{hyperref}
\usepackage{verbatim}
\usepackage{pdfpages}

\hypersetup{urlcolor=blue, colorlinks=true} 

\usepackage[inner=3cm,outer=3cm,bottom=4cm,top=4cm]{geometry}

\newtheorem{theorem}{Theorem}[section]

\newtheorem{lemma}[theorem]{Lemma}
\newtheorem{proposition}[theorem]{Proposition}
\newtheorem{corollary}[theorem]{Corollary}
\newtheorem{definition}{Definition}[section]

\theoremstyle{remark}
\newtheorem{remark}[theorem]{Remark}
\theoremstyle{definition}

\numberwithin{equation}{section}

\newcommand{\R}{\ensuremath{\mathbb{R}}}

\newcommand{\veps}{\varepsilon}
\newcommand{\Div}{\mbox{div}}

\newcommand{\plap}{\ensuremath{\Delta_p}}

\newcommand{\Imvp}{\mathcal{I}_r^p}
\newcommand{\Mmvp}{\mathcal{M}_r^p}

\newcommand{\Mmvpb}{\mathcal{M}^p}

\newcommand{\U}{\tilde{U}}

\newcommand{\dd}{\,\mathrm{d}}

\newcommand{\e}{\varepsilon}

\DeclareMathOperator{\sgn}{\textup{sign}}

\newcommand{\lap}{\Delta}

\renewcommand{\vec}{}

\begin{document}

\title[A MVF for the $p$-Laplacian]{A mean value formula for the variational $p$-Laplacian}

\author[F.~del~Teso]{F\'elix del Teso}

\address[F. del Teso]{Departamento de An\'alisis Matem\'atico y Matem\'atica Aplicada,
Universidad Complutense de Madrid, 28040 Madrid, Spain} 

\email[]{fdelteso\@@{}ucm.es}

\urladdr{https://sites.google.com/view/felixdelteso}

\keywords{$p$-Laplacian, mean value property, viscosity solutions, dynamic programming principle.}

\author[E.~Lindgren]{Erik Lindgren}

\address[E. Lindgren]{Department of Mathematics, Uppsala University, Box 480, 751 06 Uppsala, Sweden}
\email[]{erik.lindgren\@@{}math.uu.se}
\urladdr{https://sites.google.com/view/eriklindgren}

\subjclass[2010]{
 35J60, 
 35J70, 
 35J75, 
 35J92, 
 35D40, 
 35B05, 
 49L20. 
 }

\begin{abstract}\noindent   We prove a new asymptotic mean value formula for the $p$-Laplace  operator, 
$$
\plap u=\Div(|\nabla u|^{p-2}\nabla u), \quad 1<p<\infty
$$
 valid in the viscosity sense. In the plane, and for a certain range of $p$, the mean value formula holds in the pointwise sense. We also study the existence, uniqueness  and convergence of the  related dynamic programming principle.
\end{abstract}

\maketitle

\tableofcontents 

\section{Introduction} It is well known that a function is harmonic if and only if it is satisfies
$$
\fint_{B_r} \left(u(x+y)-u(x)\right) dy = 0,
$$
for all $r$ small enough. This can be relaxed: a function is harmonic at a point $x$ if and only if
$$
\fint_{B_r} \left(u(x+y)-u(x)\right) dy = o(r^2),\quad \text{as } r\to 0.
$$
In this paper, we study a new\footnote{It is new for $1<p<2$. For $p\geq 2$, it has also been found in \cite{BS18}.} asymptotic mean value property for \emph{$p$-harmonic} functions, i.e., solutions of the equation
$$
\lap_p u=0.
$$
Here $p\in (1,\infty)$ and $\lap_p$ is the $p$-Laplace operator
\begin{equation}\label{eq:plapdef}
\lap_p u = \Div(|\nabla u|^{p-2}\nabla u),
\end{equation}
the first variation of the functional
$$
u\mapsto \int_\Omega |\nabla u|^p dx.
$$
Our result implies in particular that a function is \emph{$p$-harmonic} at a point $x$ if and only if it is satisfies
$$
\fint_{B_r} |u(x+y)-u(x)|^{p-2}(u(x+y)-u(x)) dy = o(r^p),\quad \text{as } r\to 0.
$$
The major strength and novelty of our mean value formula is that it recovers the variational $p$-Laplace operator \eqref{eq:plapdef} in contrast to the other known mean value formulas that recover the normalized $p$-Laplacian,
$$
\lap_p^N  u = \frac{1}{p} |\nabla u|^{ 2-p}\lap_p u.
$$
In particular, it allows us to deal with non-homogeneous problems of the form $-\plap u=f$ with $f\not=0$, which was not possible with previous approaches.

The drawback is that it cannot be written in the form $$u(x)=A_r[u](x)+ o(r^p)$$ for some monotone operator $A_r$. However, the mean value formula is still monotonically increasing in $u$ and monotonically decreasing in $u(x)$, which is decisive in the context of viscosity solutions.

\section{Main results}
Our main results concern the asymptotic behavior as $r \to 0$  of the quantities
\[
\Imvp[\phi](x)=\frac{1}{C_{d,p}r^p} \fint_{\partial B_r} |\phi(x+y)-\phi(x)|^{p-2} (\phi(x+y)-\phi(x)) \dd \sigma(y)
\]
and
\[
\Mmvp[\phi](x)=\frac{1}{D_{d,p} r^p} \fint_{ B_r} |\phi(x+y)-\phi(x)|^{p-2} (\phi(x+y)-\phi(x)) \dd y, 
\]
where
$$
C_{d,p}= \frac12 \fint_{\partial B_1} |y_1|^{p} \dd\sigma(y)  \textup{,} \quad D_{d,p}=\frac{dC_{d,p}}{p+d}
$$
and $d$ is the dimension\footnote{$C_{d,p}$ can be expressed in terms of the so-called $\beta$-functions. We thank \'Angel Arroyo and an anonymous referee for pointing this out.}.

Our first result, that will be proved in Section \ref{sec:cons}, provides the mean value formula for $C^2$ functions. It reads:
\begin{theorem}\label{thm:MVF}
Let $p\in(1,\infty)$, $x\in \R^d$  and $\phi\in C^2(B_R(x))$ for some $R>0$. If $p\in(1,2)$ assume also that $|\nabla\phi(x)|\not=0$. Then, we have
\begin{equation*}
\Imvp[\phi](x)=\lap_p \phi(x)+o_r(1) \quad \textup{and} \quad \Mmvp[\phi](x)=\lap_p \phi(x)+o_r(1),\end{equation*}
as $r\to 0$.
\end{theorem}

The  second  of our results relates the mean value property in the viscosity sense to the $p$-Laplace equation.

\begin{theorem}\label{thm:visceq} Let $\Omega\subset  \R^d$ be bounded and open, $p\in(1,\infty)$  and $f$ be a continuous function. Then $u$ is a viscosity solution of 
$$-\Imvp[u]= f+o_r(1),\quad \text{as } r\to 0, $$ in $\Omega$ if and only if it is a viscosity solution of 
$$
-\plap u =f
$$
in $\Omega$.
\end{theorem}
 We refer to Section \ref{sec:visc} for the proof of the above result, and to Section \ref{sec:not} for the definition of viscosity solutions.

We wish to point out that for $p\geq 2$, the above results have been proved independently in \cite{BS18}, see Proposition 2.10 and Theorem 2.12 therein.

Our  third  result states that  in the plane, and for a certain range of $p$,  functions that satisfy the (homogeneous) mean value property in a pointwise sense are the same as the $p$-harmonic functions. Let $p_0$ be the root of
$$
\frac{2\left(-p+ \sqrt{(36(p-1)+(p-2)^2}\right)}{\big(-p+\sqrt{16(p-1)+(p-2)^2}\big)^2}=\frac{1}{p-1}
$$
that lies in the interval $(1,2)$. We have $p_0\approx  1.117$. 
\begin{theorem}\label{thm:pointwise} Let $\Omega \subset  \R^2$ be bounded and open, and $p\in(p_0,\infty)$. Then a continuous function $u$ satisfies 
$$-\Imvp[u]=o_r(1), \quad \text{as } r\to 0, $$  in the pointwise sense in $\Omega$ if and only if it is  a  viscosity  solution of 
$$
-\plap u =0
$$
in $\Omega$.
\end{theorem}

We refer to Section \ref{sec:plane} for the proof of this theorem and to page \pageref{lim} for a heuristic explanation on the technical limitation $p>p_0$. 

\begin{remark}
Theorem \ref{thm:visceq} and Theorem \ref{thm:pointwise} remain true if one replaces $\Imvp$ by $\Mmvp$.
\end{remark}

The fourth and the last of our main results concerns the associated dynamic programming principle. Consider the following boundary value problem
\begin{equation}\label{eq:dpp}
\begin{cases}
-\Mmvp[U_r] (x)=f(x),& x\in \Omega\\
U_r(x)=G(x), & x\in \partial \Omega_r:=\{x\in \Omega^c \ :\ \textup{dist}(x,\Omega)\leq r \}, 
\end{cases}
\end{equation}
where $G$ is a continuous extension of $g$ from $\partial \Omega$ to $\partial \Omega_r$.
\begin{theorem}
\label{thm:dpp}
Let $\Omega\subset \R^d$ be a bounded, open and $C^2$ domain,  $p\in(1,\infty)$, $f$ be a continuous function in $\overline{\Omega}$ and $g$ a continuous function on $\partial \Omega$. Then
\begin{enumerate}[\noindent i)]
\item there is a unique classical solution $U_r$ of \eqref{eq:dpp},
\item  $U_r\to u$ as $r\to 0$ uniformly in $\overline{\Omega}$, where $u$ is the viscosity solution of 
\begin{equation}\label{eq:BVP}
\begin{cases}
-\lap_p u(x) =f(x), & x\in \Omega\\
u(x)=g(x), & x\in \partial \Omega.
\end{cases}
\end{equation}
\end{enumerate}
\end{theorem}

\begin{remark}
We have stated all our results in the context of viscosity solutions. Since weak and viscosity solutions are equivalent (cf. \cite{equiv} and \cite{equivalence2}), the same results hold true for weak solutions.
\end{remark}

\section{Related results}
Recently, there has been a surge of interest around mean value properties of equations involving the $p$-Laplacian. In \cite{MRP}, it is proved that a function is $p$-harmonic if and only if 
\[
u(x)=A_r[u](x)+ o(r^2),
\]
as $r\to 0$. Here $A_r$ is the monotone operator
\[
A_r[u](x)=\frac{p-2}{2(p+d)}\left(\max_{B_r(x)}u + \min_{B_r(x)}u \right)+\frac{2+d}{p+d}\left(\fint_{B_r(x)} u(y) dy\right).
\]
This was first proved to be valid in the viscosity sense. In \cite{LM16}, this was proved to hold in the pointwise sense, in the plane and for $1<p<\hat{p}\approx 9.52$. Shortly after, this was extended to all $p\in (1,\infty)$, in \cite{AL16}.  Linked to a mean value formula, there is a corresponding dynamic programming principle (DPP), which is the solution $U_r$ of the problem $U_r=A_r[U_r]$ subject to the corresponding boundary conditions. The typical result is to show that $U_r\to u$ where $u$ is a viscosity solution of the boundary value problem associated to the $p$-Laplacian.

The above mentioned results are based on the following identity for the so-called normalized $p$-Laplacian
\begin{equation}\label{eq:plapinterp}
 \lap_p^N u := \frac{1}{p}\lap u + \frac{p-2}{p}|\nabla u|^{-2}\lap_\infty u 
\end{equation}
and the  now well-known mean value formulas for  the Laplacian and $\infty$-Laplacian. More precisely, for a smooth function $\phi$, 
\[
\frac{A_r[\phi](x)-\phi(x)}{C_{d,p}r^2} - \lap_p^N\phi(x)= o_r(1),
\]
for some constant $C_{d,p}>0$. In the last years, several other mean value formulas for the normalized $p$-Laplacian have been found, and the corresponding program (equivalence of solutions in the viscosity and classical sense and study of the associated dynamic programming principle) has been developed. See for instance \cite{AHP17},  \cite{dTMP18}, \cite{HAR16},  \cite{KMP12}, \cite{Le20},  \cite{LeMaRi19}, \cite{LeMa17},  and \cite{ObermanpLap}.

We also want to mention  \cite{GS12} and \cite{IMW17}, where two other nonlinear mean value formulas are studied, with some similarities with ours. 

It is noteworthy to mention that our results are also related to asymptotic mean value formulas for nonlocal operators involving for instance fractional or non-local versions of the $p$-Laplacian. See \cite{AMRT10} and \cite{BS18}.  In particular, in \cite{BS18}, a mean value formula and equivalence of viscosity solutions have been obtained in the case $p\geq 2$.

\section{Comments on our results}
\subsection*{Comments on Theorem \ref{thm:pointwise}} \label{lim}
The curious reader might wonder why we are not able to prove the pointwise validity for the mean value formula for the full range $p\in (1,\infty)$, as in \cite{AL16}. To make a long story short this has to do with the fact that the mean value formula considered in \cite{AL16} has quadratic scaling. It is therefore enough with an error term of order strictly larger than ${2}$. The mean value formula in the present paper however, has scaling $p/(p-1)$, which makes it necessary with an error term of order strictly larger than $p/(p-1)$. When $p<2$, this certainly comes with some difficulties that for the moment forces us to assume the larger lower bound $p>p_0$. However, we still believe that  such a result holds  in the full range $p\in(1,\infty)$.

\subsection*{Comments on the definition of viscosity solution and the proof Theorem \ref{thm:dpp}}

The operator $\lap_p\phi(x)$ is singular in the range $p\in(1,2)$ when $\nabla\phi(x)=0$. This fact forces us to choose a modified version of viscosity solution (see Definition \ref{pvisc}. As expected, when $p\geq2$ or $\nabla\phi(x)\not=0$, this definition is equivalent to the usual one (cf. \cite{AR18}). 

This definition of viscosity solution adds some extra technicalities in the proofs of this manuscript. In particular in the proof of convergence of Theorem \ref{thm:dpp}. Here we follow the classical program developed in \cite{BS91} and adapted to the context of \em homogeneous \em problems involving the $p$-Laplacian.

\subsection*{Comments on the limit $p\to 1$} Formally, when $p=1$ the mean value formula becomes
$$
\fint_{B_r} \sgn(u(x+y)-u(x)) dy = o(r),\quad \text{as } r\to 0,
$$
or
$$
\frac{1}{r}\left(\frac{|\{y\in B_r \,: \, u(x+y)>u(x)\} |}{|B_r|}-\frac{|\{y\in B_r \,: \, u(x+y)<u(x)\} |}{|B_r|}\right) = o_r(1),
$$
which could relate to $1$-harmonic functions. We plan to study this possibility in the future.

 \subsection*{More general datum} It would also be interesting to study problems where $f=f(x,u,\nabla u(x))$ has the right monotonicity assumptions as described in \cite{MO}. Theorem \ref{thm:visceq} follows in a straightforward way. However, the convergence of dynamic programming principles like in Theorem \ref{thm:dpp} would require a more delicate study, both in terms of existence and properties of the $r$-scheme, and the study of convergence based on the Barles-Souganidis approach. 

\subsection*{Plan of the paper}
The plan of the paper is as follows. In Section \ref{sec:not}, we introduce some notation
and the notions of viscosity solutions.  This is followed by Section \ref{sec:cons}, where we prove the mean value formula for $C^2$ functions. This result is then used in Section \ref{sec:visc}, where we prove the mean value formula for viscosity solutions. In Section \ref{sec:plane}, we prove that in dimension $d=2$, and for a certain range of $p$, functions that satisfy the (homogeneous) mean value property in a pointwise sense are the same as the $p$-harmonic functions. In Section \ref{sec:dpp}, we study existence, uniqueness and convergence for the dynamical programming principle. Finally, in the Appendix, we prove and state some auxiliary inequalities.

\section{Notation and prerequisites}\label{sec:not}Throughout this paper, $d$ will denote the dimension and we will for $p\in (1,\infty)$ use the notation
$$
J_p(t)=|t|^{p-2}t.
$$
We now define viscosity solutions of the related equations and mean value properties. We adopt the definition of solutions from \cite{equivalence2}. 
 \begin{definition} [Viscosity solutions of the equation] \label{pvisc} Suppose that $f$
 is continuous function in $\Omega$. We say that a lower (resp. upper) semicontinuous function $u$ in $\Omega$ is a \emph{viscosity supersolution} (resp. \emph{subsolution})  of the
  equation $$-\plap u=f $$  in $\Omega$ if the
  following holds: whenever $x_0 \in \Omega$ and $\varphi \in
  C^2(B_R(x_0))$ for some $R>0$ are such that $|\nabla \varphi(x)|\neq 0$ for $x\in B_R(x_0)\setminus\{x_0\}$, 
$$ \varphi(x_0) = u(x_0) \quad \text{and} \quad \varphi(x) \leq
u(x) \quad  \text{(resp. $\varphi(x)\geq u(x)$)} \quad \text{for all} \quad
x \in  B_R(x_0)\cap \Omega,$$
then we have
\begin{equation}\label{eq:testfuncdef}\lim_{\rho\to 0}\sup_{B_\rho(x_0)\setminus\{x_0\}}\left(-\plap \varphi (x)\right)\geq f(x_0) \quad   \text{(resp. $\lim_{\rho\to 0}\inf_{B_\rho(x_0)\setminus\{x_0\}}\left(-\plap \varphi (x)\right)\leq f(x_0)$)}.
\end{equation}
 A \emph{viscosity solution} is a continuous function being
both a viscosity supersolution and a viscosity subsolution.
\end{definition}
 
\begin{remark}\label{rem:equivviscform}
We consider condition \eqref{eq:testfuncdef} to avoid problems with the definition of $-\Delta_p \phi(x_0)$ when $\nabla \varphi(x_0)=0$ and $p\in(1,2)$. However, when either $p\geq2$ or $\nabla \varphi(x_0)\not=0$, \eqref{eq:testfuncdef} can be replaced by the standard one, i.e.,
\begin{equation}\label{eq:testfuncdef2}
-\plap \varphi (x_0)\geq f(x_0) \quad   \text{(resp. $-\plap \varphi (x_0)\leq f(x_0)$)}.
\end{equation}
\end{remark}

\begin{definition} [The mean value property in the viscosity sense]  Suppose that $f$
 is continuous function in $\Omega$. We say that a lower (resp. upper) semicontinuous function $u$ in ${\Omega}$ is a \emph{viscosity supersolution} (resp. \emph{subsolution}) of the
  equation $$-\Imvp[u]=f+o_r(1) $$  in $\Omega$ if the
  following holds: whenever $x_0 \in \Omega$ and $\varphi \in
  C^2(B_R(x_0))$ for some $R>0$ are such that $|\nabla \varphi(x)|\neq 0$ for $x\in B_R(x_0)\setminus\{x_0\}$, 
$$ \varphi(x_0) = u(x_0) \quad \text{and} \quad \varphi(x) \leq
u(x) \quad  \text{(resp. $\varphi(x)\geq u(x)$)} \quad \text{for all} \quad
x \in  B_R(x_0)\cap \Omega,$$
then we have 
$$ \lim_{\rho\to 0}\sup_{B_\rho(x_0)\setminus\{x_0\}}\left(-\Imvp[\varphi](x)\right)\geq f(x_0) +o_r(1) \quad  \text{(resp. $ \lim_{\rho\to 0}\inf_{B_\rho(x_0)\setminus\{x_0\}}\left(-\Imvp[\varphi](x)\right)\leq f(x_0) +o_r(1) $)}.$$
A \emph{viscosity solution} is a continuous function being both a viscosity supersolution and a viscosity subsolution.
\end{definition}

\begin{remark}
The above definition can also be considered with $\Mmvp$ instead of $\Imvp$.
\end{remark}

Finally, we define the concept of viscosity solution for the the boundary value problem \eqref{eq:BVP}.
\begin{definition}[Viscosity solutions of the boundary value problem] 
Suppose that $f$ is continuous function in $\overline{\Omega}$, and that $g$ is a continuous function in $\partial \Omega$. We say that a lower (resp. upper) semicontinuous function $u$  in $\overline{\Omega}$  is a \emph{viscosity supersolution} (resp. \emph{subsolution}) of \eqref{eq:BVP} if
\begin{enumerate}[(a)]
\item $u$ is a viscosity supersolution (resp. subsolution) of $-\Delta_p u=f$ in $\Omega$ (as in Definition \ref{pvisc});
\item $u(x)\geq g(x)$ (resp. $u(x)\leq g(x)$) for $x\in \partial \Omega.$
\end{enumerate}
A \emph{viscosity solution} of \eqref{eq:BVP} is a continuous function in $\overline{\Omega}$ being both a viscosity supersolution and a viscosity subsolution.
\end{definition}

\section{The mean value formula for $C^2$-functions}\label{sec:cons}
In this section we prove the mean value formulas for $C^2$-functions  as presented in Theorem \ref{thm:MVF}. The proof is split into two different cases: $p>2$ and $p<2$. The case $p=2$ is well known so we leave that out.  We restate the results for convenience.
\begin{theorem}\label{thm:consistency}
Let $p\in(2,\infty)$ and $\phi\in C^2(B_R(x))$ for some $R>0$. Then
\[
\frac{1}{C_{d,p}r^p} \fint_{\partial B_r} |\phi(x+y)-\phi(x)|^{p-2} (\phi(x+y)-\phi(x)) \dd \sigma(y)=\plap \phi(x)+ o_r(1), 
\]
where $ C_{d,p}=\frac12 \fint_{\partial B_1} |y_1|^{p} \dd\sigma(y)$.
\end{theorem}
\begin{proof}
Since $\phi \in C^2$ near $x$, we have that
\[
\phi(x+y)-\phi(x)= y\cdot \nabla \phi(x) + \frac12 y^{T} D^2 \phi(x) y +  o(|y|^2).
\]
Using Lemma \ref{lem:pineq1} for $\e=0$ and with $a=y\cdot \nabla\phi(x) +\frac12  y^{T} D^2 \phi(x) y$ and $b= o(|y|^2)$ we get 
\[
\begin{split}
 J_p(\phi(x+y)-\phi(x))&=J_p(y\cdot \nabla\phi(x) +\frac12  y^{T} D^2 \phi(x) y +  o(|y|^2))\\
&=J_p(y\cdot \nabla\phi(x) + \frac12 y^{T} D^2 \phi(x) y)+  o(|y|^p).
\end{split}
\]
Therefore, 
\begin{equation}
\label{eq1}
\begin{split}
 A_r:=& \fint_{\partial B_r} J_p(\phi(x+y)-\phi(x)) \dd \sigma(y)\\
=& \fint_{\partial B_r} J_p(  y\cdot \nabla\phi(x) +\frac12  y^{T} D^2 \phi(x) y) \dd \sigma(y) +  o(r^p).
  \end{split}
\end{equation}
Now we use Lemma \ref{lem:pineq1} for some $\e\in(0, p-2)$ with $a=y\cdot \nabla \phi(x)$ and $b=\frac12  y^T D^2\phi(x) y$ and obtain 
\[
\begin{split}
J_p(  y\cdot \nabla\phi(x) +\frac12  y^{T} D^2 \phi(x) y)=&|y\cdot \nabla \phi(x)|^{p-2}y\cdot \nabla \phi(x)+(p-1)|y\cdot \nabla \phi(x)|^{p-2} \frac12 y^T D^2 \phi(x) y\\&+\underbrace{\mathcal{O}(|y|^{p-2-\veps})\mathcal{O}(y^{2(1+\veps)})}_{o(|y|^p)}.
\end{split}
\]
Since the first term is odd and we are integrating over a sphere in \eqref{eq1}, we get
\[
A_r=\frac12 (p-1)\fint_{\partial B_r} |y\cdot \nabla \phi(x)|^{p-2} y^T D^2 \phi(x) y \dd \sigma(y) + o(r^p).	
\]
Without loss of generality, assume that, $\nabla \phi(x)=c \vec{e}_1$ for some $c\geq 0$.   Note that this assumption implies that $|\nabla \phi(x)|=c$ and $\Delta_\infty \phi(x)= c^2 D_{11}\phi(x)$.  The symmetry of the integral  and the term $y^T D^2 \phi(x) y$ imply that
\begin{equation*}\label{eq:Araux}
\begin{split}
A_r&=\frac12 c^{p-2}(p-1)\fint_{\partial B_r}|y_1|^{p-2} \left( \sum_{i=1}^d y_i^2 D_{ii}\phi(x) \right)\dd \sigma(y)+ o(r^p)\\
&= \frac12 c^{p-2}(p-1) \left(\sum_{i=1}^d D_{ii}\phi(x) \fint_{\partial B_r} |y_1|^{p-2} y_i^2 \dd \sigma(y)\right)+o(r^p).
\end{split}
\end{equation*}
Note that  if $d\geq2$,  for all $i\not=1$, integration by parts implies
\[
 {C}_{d,p} r^{p}= \frac12 \fint_{\partial B_r} |y_1|^{p} \dd\sigma(y)=\frac 12(p-1)\fint_{\partial B_r} |y_1|^{p-2}  y_i^2 \dd \sigma(y).
\]
Thus,
\[
\begin{split}
A_r&=  {C}_{d,p} r^{p}  c^{p-2} \left( (p-1) D_{11}\phi(x) +  \sum_{i=2}^d D_{ii}\phi(x)\right)+o(r^p)\\
&= {C}_{d,p} r^{p}  c^{p-2} \left(   \sum_{i=1}^d D_{ii}\phi(x) +(p-2) D_{11}\phi(x) \right)+o(r^p)\\
&=  {C}_{d,p} r^{p}\left( |\nabla \phi(x)|^{p-2} \Delta \phi(x)+(p-2) |\nabla \phi(x)|^{p-4} \Delta_\infty \phi(x)\right)+o(r^p).
\end{split}
\]
Now, from identity \eqref{eq:plapinterp} we get
\[
\begin{split}
\Imvp[\phi](x)&= \frac{1}{C_{d,p}r^p}A_r\\
&=  |\nabla \phi(x)|^{p-2} \Delta \phi(x)+(p-2) |\nabla \phi(x)|^{p-4} \Delta_\infty \phi(x)+o_r(1)\\
&=   \Delta_p \phi(x) + o_r(1),
\end{split}
\]
which concludes the proof. 
\end{proof}

We now proceed to the case $p<2$, which is slightly more involved.

\begin{theorem}\label{thm:consistency2}
Let $p\in(1,2)$ and $\phi\in C^2(B_R(x))$ for some $R>0$. Assume also that $|\nabla \phi(x)|\neq 0$. Then
\[
\frac{1}{C_{d,p}r^p} \fint_{\partial B_r} |\phi(x+y)-\phi(x)|^{p-2} (\phi(x+y)-\phi(x)) \dd \sigma(y)=\plap \phi(x)+ o_r(1), 
\]
where $C_{d,p}=\frac12 \fint_{\partial B_1} |y_1|^p \dd\sigma(y).$
\end{theorem}

\begin{proof}   We keep the notation $A_r$ of \eqref{eq1}. Without loss of generality, we assume that $\nabla \phi(x)=c\vec{e}_1$ for some $c>0$. We split the proof into several parts. 

{\it Part 1:}  First we prove an estimate that will be used several times along the proof. Let $\alpha\in(0,1)$ and $\rho\geq0$ small enough. Then
\begin{equation}\label{eq:keyest}
\fint_{\partial B_r} \left|\frac{z}{|z|}\cdot \nabla\phi(x)+\rho \left(\frac{z}{|z|}\right)^T D^2\phi(x) \left(\frac{z}{|z|}\right)\right|^{-\alpha} \dd \sigma(z)\leq C_1
\end{equation}
for some $C_1=C_1(\alpha,d)\geq0$.  To prove \eqref{eq:keyest}, we first note that its left hand side is equal to
\[
C_2 c^{-\alpha}\int_{\partial B_1} |z \vec{e}_1+ \rho c^{-1} z^T D^2\phi(x) z|^{-\alpha} \dd \sigma(z)
\]
for some constant $C_2=C_2(d)>0$. Estimate \eqref{eq:keyest} follows from applying Lemma \ref{lem:stupid} with $L(\omega,\omega)= \rho c^{-1} \omega^T D^2\phi(x) \omega$ choosing $\rho$ small enough such that \eqref{eq:QFest} holds.

{\it Part 2:} In this part, we prove 
\begin{equation*}\label{eq:ArV1}
\begin{split}
A_{r}=& \fint_{\partial B_r} J_p(\phi(x+y)-\phi(x)) \dd \sigma(y) \\
=& \fint_{\partial B_r} J_p(y\cdot \nabla\phi (x) +\frac12 y^T D^2 \phi(x) y) \dd \sigma(y) + o(r^p).
\end{split}
\end{equation*}
By Taylor expansion,
\begin{equation*}\label{eq:ArV2}
A_{r}= \fint_{\partial B_r} J_p(y\cdot \nabla\phi (x) +\frac12 y^T D^2 \phi(x) y+ o(|y|^2)) \dd \sigma(y).
\end{equation*}
Lemma \ref{lem:pineq3} with $a=y\cdot \nabla\phi (x) +\frac12 y^T D^2 \phi(x) y$ and $b=o(|y|^2)$ implies
\[
\begin{split}
\big|J_p(y\cdot \nabla\phi (x) +\frac12 y^T D^2 &\phi(x) y+ o(|y|^2))-J_p(y\cdot \nabla\phi (x) +\frac12 y^T D^2 \phi(x) y)\big|\\
&\leq C\left(|y\cdot \nabla\phi (x) +\frac12 y^T D^2 \phi(x) y|+|o(|y|^2)|\right)^{p-2}o(|y|^2)\\
&\leq C\left|y\cdot \nabla\phi (x) +\frac12 y^T D^2 \phi(x) y\right|^{p-2}o(|y|^2)\\
&= C\left|\hat{y}\cdot \nabla\phi (x) +\frac12 |y|\hat{y}^T D^2 \phi(x) \hat{y}\right|^{p-2}o(|y|^p),
\end{split}
\]
where $\hat{y}:=y/|y|$. Thus,
\[
\begin{split}
\bigg|A_{r}- \fint_{\partial B_r} J_p(y\cdot \nabla\phi (x) +\frac12 y^T D^2 \phi(x) y) \dd \sigma(y) \bigg|&\leq o(r^p) \fint_{\partial B_r} \left|\hat{y}\cdot  \nabla \phi (x) +\frac12 r  \hat{y}^T D^2 \phi(x)\hat{y}\right|^{p-2} \dd \sigma(y)\\&=o(r^p)\\
\end{split}
\]
where the last identity follows from applying \eqref{eq:keyest} with $\rho=r$ (choosing $r$ small enough).

{\it Part 3: } This part amounts to proving that 
\begin{equation*}
\label{eq:best}
|B_{r,\gamma}|:=\Big|\fint_{\partial B_r\cap \{|\hat y\cdot e_1|\leq \gamma\}} J_p(y\cdot \nabla\phi (x) +\frac12 y^T D^2 \phi(x) y) \dd \sigma(y)\big|\leq C_\gamma r^p,
\end{equation*}
where $C_\gamma\to 0$ as $\gamma\to 0$. First we note that with our notation we have
\begin{equation}\label{eq:identity10}
\begin{split}
J_p(y\cdot \nabla\phi (x) +y^T D^2 \phi(x) y)&=J_p(c y \cdot \vec{e}_1 +\frac12 y^T D^2 \phi(x) y)\\
&=(c|y|)^{p-1} J_p(\hat{y} \cdot \vec{e}_1 +\frac12 c^{-1}|y|\hat{y}^T D^2 \phi(x) \hat{y} ).
\end{split}
\end{equation}
Lemma \ref{lem:pineq3} with $a=\hat{y} \cdot \vec{e}_1$ and $b=\frac12  c^{-1}|y|\hat{y}^T D^2 \phi(x) \hat{y}$ implies
 \[
\begin{split}
\Big|(c|y|)^{p-1} &J_p(\hat{y} \cdot \vec{e}_1 +\frac12 c^{-1}|y|\hat{y}^T D^2 \phi(x) \hat{y} )-(c|y|)^{p-1} J_p(\hat{y} \cdot \vec{e}_1)\Big|\\
&\leq C(c|y|)^{p-1}\left(|\hat{y} \cdot \vec{e}_1|+\frac12  c^{-1}|y||\hat{y}^T D^2 \phi(x) \hat{y}|\right)^{p-2}\frac12 c^{-1}|y||\hat{y}^T D^2 \phi(x) \hat{y}|\\
&\leq C|y|^p|\hat{y} \cdot \vec{e}_1|^{p-2}.
\end{split}
\]
By antisymmetry
\[
\fint_{\partial B_r\cap \{|\hat y\cdot e_1|\leq\gamma\}}(c|y|)^{p-1}J_p(\hat{y} \cdot \vec{e}_1) \dd \sigma(y) = 0.
\]
This, \eqref{eq:keyest} with $\alpha = (p-2)(1+\delta)>-1$  and $\rho=0$, and H\"older's inequality imply
\[
\begin{split}
|B_{r,\gamma}|&\leq Cr ^p\fint_{\partial B_r}|\hat{y} \cdot \vec{e}_1|^{p-2}\chi_{ \{|\hat y\cdot e_1|\leq \gamma\}} \dd\sigma (y)\\
&\leq Cr^p\left(\fint_{\partial B_r}|\hat{y} \cdot \vec{e}_1|^{ \alpha} \dd\sigma (y)\right)^\frac{1}{1+\delta} \left(\fint_{\partial B_r} \chi_{ \{|\hat y\cdot e_1|\leq \gamma\}}\dd\sigma (y)\right)^{\frac{\delta}{1+\delta}}\\
&\leq CC_\gamma r^p,
\end{split}
\]
where
$$
C_\gamma = \left(\fint_{\partial B_r} \chi_{ \{|\hat y\cdot e_1|\leq\gamma\}}\dd\sigma (y)\right)^{\frac{\delta}{1+\delta}}\stackrel{\gamma\to0}{\longrightarrow} 0.
$$

{\it Part 4: } We will now prove that for fixed $\gamma>0$, 
\begin{equation}
\label{eq:brgamma}
\begin{split}
D_{r,\gamma}&:=\fint_{\partial B_r\cap \{|\hat y\cdot e_1|>\gamma\}} J_p(y\cdot \nabla\phi (x)  +\frac12 y^T D^2 \phi(x) y) \dd \sigma(y)\\
&=\frac12 \fint_{\partial B_r\cap \{|\hat y\cdot e_1|>\gamma\} } (p-1)|y\cdot \nabla \phi(x)|^{p-2} y^T D^2 \phi(x) y \dd \sigma(y) + o(r^p).
\end{split}
\end{equation}
Here it is crucial that the integrals are restricted to the set $ \{|\hat y\cdot e_1|>\gamma\}$.

We observe that outside $\xi=0$ the function $\xi\mapsto J_p(\xi)$ is smooth. In particular, for $a\not=0$ and $b$ such that $|b|<|a|/2$, we have the following estimate
\[
|J_p(a+b)-J_p(a)- J_p'(a)b|\leq C (|a+b|^{p-2-\delta}+|a|^{p-2-\delta})|b|^{1+\delta}
\]
for any $\delta \in (0,p-1)\subset(0,1)$.  For any  $y$ such that $|\hat y \cdot \vec{e}_1|>\gamma $, the above estimate with $a=\hat{y} \cdot \vec{e}_1 $ and $b=\frac12 c^{-1}|y|\hat{y}^T D^2 \phi(x) \hat{y}$ (since $a\not=0$ and $b<\gamma/2<|a|/2$ by choosing $r=|y|$ small enough),  together with \eqref{eq:identity10} imply 
\[
\begin{split}
J_p(  y\cdot \nabla\phi(x) +\frac12  y^{T} D^2 \phi(x) y)=&(c|y|)^{p-1} |\hat{y}\cdot \vec{e}_1|^{p-2}\hat{y}\cdot \vec{e}_1 +(p-1)|y\cdot \nabla \phi(x)|^{p-2}\frac12  y^T D^2 \phi(x) y+R(y), 
\end{split}
\]
where $R(y)$ is bounded by 
\[\begin{split}
 C_3|y|^{p+\delta}  \left(\left|\hat{y} \cdot \vec{e}_1 +\frac12 c^{-1}|y|\hat{y}^T D^2 \phi(x) \hat{y}\right|^{p-2-\delta} + |\hat{y} \cdot \vec{e}_1 |^{p-2-\delta}\right)\left|\frac12 \hat{y}^T D^2 \phi(x) \hat{y}\right|^{1+\delta}\\
 \leq C_4 |y|^{p+\delta}  \left(\left|\hat{y} \cdot \vec{e}_1 +\frac12 c^{-1}|y|\hat{y}^T D^2 \phi(x) \hat{y}\right|^{p-2-\delta} + |\hat{y} \cdot \vec{e}_1 |^{p-2-\delta}\right).
\end{split}
\]
For some constants $C_3,C_4\geq0$ and $r$ small enough (depending on $\gamma$). Moreover, by antisymmetry, 
$$
\fint_{\partial B_r\cap \{|\hat y\cdot e_1|>\gamma\}}(c|y|)^{p-1}|\hat y_1|^{p-2}\hat y_1 \dd \sigma(y) = 0.
$$
We apply \eqref{eq:keyest} with $\alpha= -p+2+\delta \in (0,1)$ two times, first with $\rho=r$ and later with $\rho=0$ to get  
\[
\fint_{\partial B_r\cap \{|\hat y\cdot e_1|>\gamma\}} |R(y)|\dd \sigma(y)\leq \fint_{\partial B_r} |R(y)|\dd \sigma(y)\leq O(r^{p+\delta})=o(r^p)
\]
where the bound is uniform for fixed $\gamma$. This implies \eqref{eq:brgamma}.

{\it Part 5:}   From  parts 2 and 3  we have
$$
\limsup_{r\to 0} \frac{|A_r-D_{r,\gamma}|}{r^p}\leq  \limsup_{r\to 0} \frac{|B_{r,\gamma}|}{r^p}\leq C_\gamma \stackrel{\gamma\to0}{\longrightarrow} 0,
$$
Moreover, by part 4
\[
\begin{split}
\lim_{r\to 0} \frac{D_{r,\gamma}}{r^p}&=  \frac12 \lim_{r\to 0} r^{-p}  \fint_{\partial B_r\cap \{|\hat{y}\cdot e_1|>\gamma\} } (p-1)|y\cdot \nabla \phi(x)|^{p-2} y^T D^2 \phi(x) y \dd \sigma(y) \\
&=\frac12 (p-1)\fint_{\partial B_1\cap \{| z \cdot e_1|>\gamma\} } |z\cdot \nabla \phi(x)|^{p-2} z^T D^2 \phi(x) z \dd \sigma(z).
\end{split}
\]
Since the last term is independent of $r$ and converges to
\[
\frac12 (p-1)\fint_{\partial B_1} |z\cdot \nabla \phi(x)|^{p-2} z^T D^2 \phi(x) z \dd \sigma(z)= C_{d,p}\lap_p \phi (x), 
\]
as $\gamma\to 0$, where the last equality follows from the proof of Theorem \ref{thm:consistency}, the result follows.
\end{proof}

As an immediate corollary, we obtain that also the mean over balls have the same asymptotic limit.

\begin{corollary}\label{cor:consistency}
Let $p\in(1,\infty)$ and $\phi\in C^2(B_R(x))$. If $p<2$, assume also that $|\nabla \phi(x)|\neq 0$. Then
\[
\frac{1}{D_{d,p}r^p} \fint_{B_r} |\phi(x+y)-\phi(x)|^{p-2} (\phi(x+y)-\phi(x)) \dd y=\plap \phi(x)+ o_r(1),
\]
where $D_{d,p}=\frac{dC_{d,p}}{p+d}.$
\end{corollary}

\section{Viscosity solutions}\label{sec:visc}

Now we prove that satisfying the asymptotic mean value property in the viscosity sense is equivalent to being a viscosity solution of the corresponding PDE.

\begin{proof}[~Proof of Theorem \ref{thm:visceq}] We only prove that the notion of supersolutions are equivalent. The case of a subsolution can be treated similarly.
Suppose first that $u$ is a viscosity supersolution of $-\plap u =f$ in $\Omega$. Take $x_0 \in \Omega$ and $\varphi \in
  C^2(B_R(x_0))$ for some $R>0$ such that $|\nabla \varphi(x)|\neq 0$ when $x\neq x_0$, 
$$ \varphi(x_0) = u(x_0) \quad \text{and} \quad \varphi(x) \leq
u(x)\quad \text{for all} \quad
x \in \Omega.$$
Since $u$ is viscosity supersolution of $-\plap u =f$ we have that for given $\e>0$ there is $x\in B_\rho(x_0)\setminus\{x_0\}$ with $\rho=\rho(\e)$ such that 
$$- \plap \varphi (x)\geq f(x_0)-\e.$$
By Theorem \ref{thm:MVF}
$$
\plap \varphi (x)=\Imvp[\varphi](x)+o_r(1).
$$
Therefore, 
$$
-\Imvp[\varphi](x)\geq f(x_0)+o_r(1)-\e.
$$
Since $\e$ was arbitrary, this proves the mean value supersolution property.
Now suppose instead that $u$ is a viscosity supersolution of 
$$-\Imvp[u]=f+o_r(1) $$ in $\Omega$. Take again $x_0 \in \Omega$ and $\varphi \in
  C^2(B_R(x_0))$ for some $R>0$ such that $|\nabla \varphi(x)|\neq 0$ when $x\neq x_0$,
$$ \varphi(x_0) = u(x_0) \quad \text{and} \quad \varphi(x) \leq
u(x)\quad \text{for all} \quad
x \in \Omega.$$
By the definition of a supersolution, for given $\e>0$ there is $x\in B_\rho(x_0)\setminus\{x_0\}$ with $\rho=\rho(\e)$ such that 
$$
-\Imvp[\varphi](x)\geq f(x_0)+o_r(1)-\e.
$$
Again by Theorem \ref{thm:MVF} 
$$
\plap \varphi (x)=\Imvp[\varphi](x)+o_r(1), 
$$
which implies
$$
-\plap \varphi (x)\geq f(x_0)+o_r(1)-\e. 
$$
Passing $r\to 0$ implies $-\plap\varphi(x)\geq f(x_0)-\e$. Again, since $\e$ was arbitrary, the proof is complete.
\end{proof}

\section{The pointwise property in the plane}\label{sec:plane}
Now we are ready to prove that the mean value property is satisfied in a pointwise sense in the aforementioned range of $p$.

\begin{proof}[~Proof of Theorem \ref{thm:pointwise}] 

Assume that  $u$ satisfies  
 \begin{equation}\label{eq:MVPhomo}-\Imvp[u]=o_r(1)
\end{equation}   in the pointwise sense  in $\Omega$. Then it is obviously also a viscosity solution.
By Theorem \ref{thm:visceq} it is also a viscosity solution of $-\plap u =0$ which proves the first implication.

Assume now instead that $u$ is a viscosity solution of $-\plap u =0$ and let  $x_0\in \Omega$. If $|\nabla u(x_0)|\neq 0$, then $u$ is real analytic near $x_0$ and the mean value formula holds trivially at $x_0$ by Theorem \ref{thm:MVF}. If $|\nabla u(x_0)|=0$ we need different arguments depending on $p$.

 \textbf{Case $\mathbf{p\geq 2}$: }  The case $p=2$ is well-known and we do not comment on it. If $p>2$,  Theorem 1 in \cite{IM} implies that  $u\in C^{1,\alpha}$ for some $1>\alpha>1/(p-1)$. Then
$$
|u(x_0+y)-u(x_0)|\leq C|y|^{1+\alpha},
$$
 which implies that 
\begin{equation}\label{eq:MVPigual0}
|\Imvp[u](x_0)|\leq Cr^{-p}r^{(p-1)(1+\alpha)}=o_r(1),
\end{equation}
which ends the proof in this case.
\medskip

 \textbf{Case $\mathbf{p_0<p< 2}$: } First we use that on page 146 in \cite{LM16} it is proved that  for some integer $n\geq1$  we have that 
$$
|D^2 u|=\mathcal{O}\left(r^{\eta_n-1}\right) \quad \textup{in} \quad B_r(x_0)
$$
where 
$$
\frac{1}{\eta_n}:=\frac12\left(-p+\sqrt{4\left(1+\frac1n\right)^2(p-1)+(p-2)^2}\right).
$$
In particular, when $n\geq 3$ we have that $ 1/\eta_n<p-1$ which implies  
$|D^2 u|=o\big(r^{\frac{1}{p-1}-1}\big)$  in $B_r(x_0)$. By Taylor expansion we thus get, 
\[
|u(x_0+y)-u(x_0)|^{p-1}\leq (\|D^2u\|_{L^\infty(B_r(x_0))}r^2)^{p-1}= o(r^{\frac{1}{p-1}+1})^{p-1}=o(r^p)
\]
which in turn implies   $-\Imvp[u](x_0) =o_r(1)$ as in \eqref{eq:MVPigual0}.

We still need to check the cases $n=1$ and $n=2$. We do it in several steps.

\textit{Step 1. } For this we need a refined expansion around a critical point  $x_0$ (and assume $u(x_0)=0$ for simplicity)  taken  from  pages 147-148 in \cite{LM16}.  It reads
\begin{equation}
\label{eq:hodoexp}
u(x)=\mathfrak{A}(x)+\mathcal{O}\left(r^{\gamma} \right)\quad \textup{for all} \quad x\in B_r(x_0),
\end{equation}
with
\begin{equation}\label{eq:defgamma}
 \gamma= 1+\frac{\lambda^{(n)}_{n+2}} {  \big(\lambda^{(n)}_{n+1}\big)^2}
 \quad \textup{and} \quad \lambda_k^{(n)}=\frac12\left( -np+\sqrt{4k^2(p-1)+n^2(p-2)^2}\right)
\end{equation}
and where the function $\mathfrak{A}(x)$ is defined by (see pages 3864-3865 in \cite{AL16})
$$
\mathfrak{A}=\widetilde{\mathfrak{A}}\circ \mathcal{A}^{-1}.
$$
Here $\mathcal{A}$ and $\widetilde{\mathfrak{A}}$ are defined in complex variables by
$$
\mathcal{A}(re^{i\theta})=r^\beta e^{-in\theta}\left(e^{i(n+1)\theta}+\varepsilon e^{-i(n+1)\theta}\right)
$$
$$
|\mathcal{A}(re^{i\theta})|=r^\beta m(\theta),\quad m(\theta)= \sqrt{1+\varepsilon^2+2\varepsilon\cos(2(n+1)\theta)},
$$
and
$$
\widetilde{\mathfrak{A}}(re^{i\theta})=Cr^{\alpha}\cos((n+1)\theta).
$$
In the above, $C$, $\alpha$, $\beta$ and $\varepsilon$ are constants depending on $n$, but their values will not be important in what follow, except the fact that $|\varepsilon|<(2n+1)^{-1}$, see equation (2.4) on page 3861 in \cite{AL16}. Note that by \eqref{eq:hodoexp}, we necessarily have
$$
\mathfrak{A} (x_0)=|\nabla \mathfrak{A} (x_0)|=0.
$$

\textit{Step 2.} 
We prove now  that $\mathfrak{A}$ satisfies the mean value property,  i.e.
\[
\Imvp[\mathfrak{A}](x_0)=0.
\]
We define, 
$$
\widetilde{B}_R = {\mathcal{A}}^{-1}(B_R)=\left\{re^{i\theta}: \, r^\beta<\frac{R}{m(\theta)}\right\}, 
$$
where the equality follows from the fact that $|\mathcal{A}(re^{i\theta})|=r^\beta m(\theta)$. We also compute the jacobian of $\mathcal{A}$ and find
$$J(r e^{i\theta})=|D\mathcal{A}|(r e^{i\theta})= \beta r^{2(\beta-1)}(1-(2n+1)\varepsilon^2-2n\varepsilon\cos(2(n+1)\theta))>0,
$$
where we used that $|\varepsilon|<(2n+1)^{-1}$. By a change of variables
\[
\begin{split}
&\int_{B_R} |{\mathfrak{A}}(re^{i\theta})|^{p-2}{\mathfrak{A}}(re^{i\theta}) dA  = \int_{\tilde{B}_R} |\widetilde{\mathfrak{A}}(re^{i\theta})|^{p-2}\widetilde{\mathfrak{A}}(re^{i\theta})J(re^{i\theta}) r dr d\theta\\
&=C^{p-1}\beta\int_0^{2\pi}\int_0^{r(\theta)} r^{\alpha(p-1)+2\beta-1} |\cos((n+1)\theta)|^{p-2}\cos((n+1)\theta)j(\theta) dr d\theta
\end{split}
\]
where
$$
r(\theta)=\left(\frac{R}{m(\theta)} \right)^\frac{1}{\beta}, \quad j(\theta)= 1-(2n+1)\varepsilon^2-2n\varepsilon\cos(2(n+1)\theta).
$$
Hence, we see that we are left with an integral of the form
$$
\int_0^{2\pi} f(\cos(2(n+1)\theta))|\cos((n+1)\theta)|^{p-2}\cos((n+1)\theta)d\theta.
$$
By change of variables we can reduce this to computing
$$
\int_0^{2\pi} f(\cos(2\theta))|\cos(\theta)|^{p-2}\cos(\theta)d\theta = 0,
$$
by symmetry. Therefore, 
$$
\int_{B_R} |{\mathfrak{A}}(re^{i\theta})|^{p-2}{\mathfrak{A}}(re^{i\theta}) dA = 0
$$
and $\mathfrak{A}$ satisfies the mean value property. 

\textit{Step 3. }Now we go back to $u$. Using \eqref{eq:hodoexp}, we have together with Lemma \ref{lem:pineq3}
\[ 
\begin{split}
&|J_p(u(x_0+y)-u(x_0))-J_p({\mathfrak{A}}(x_0+y)-{\mathfrak{A}}(x_0))| =\mathcal{O}\big(r^{(p-2)\gamma} \big)\mathcal{O}\big(r^\gamma\big)=\mathcal{O}\big(r^{(p-1)\gamma} \big),
\end{split}
\]
with $\gamma$ given in \eqref{eq:defgamma}. By Step 2, $\mathfrak{A}$ satisfies the mean value property at $x_0$ and thus
$$
|\Imvp[u](x_0)|\leq Cr^{-p+(p-1)\gamma}.
$$
The proof will be finished if we verify that $\gamma>p/(p-1)$, that is,
\begin{equation}\label{eq:lastcheck}
\frac{\lambda^{(n)}_{n+2}} {  \big(\lambda^{(n)}_{(n+1)}\big)^2}>\frac{1}{p-1}.
\end{equation}

First we verify   \eqref{eq:lastcheck} when $n=1$. In this case 
$$
\lambda_k^{(1)} = \frac12\big(-p+\sqrt{4k^2(p-1)+(p-2)^2}\big)
$$
so that \eqref{eq:lastcheck}
becomes 
$$
\frac{2\left( -p +\sqrt{36(p-1)+(p-2)^2}\right)}{\big(-p + \sqrt{16(p-1)+(p-2)^2}\big)^2}>\frac{1}{p-1}.
$$
This inequality is exactly true  when $p\in (p_0,2)$.

If $n=2$ then 
$$ 
\lambda_k^{(2)} = \frac12\left(-2p+\sqrt{4k^2(p-1)+4(p-2)^2}\right)
$$
and \eqref{eq:lastcheck} becomes
$$ 
\frac{2\left(-2p+\sqrt{64(p-1)+4(p-2)^2}\right)}{\big(-2p+\sqrt{36(p-1)+4(p-2)^2}\big)^2}>\frac{1}{p-1}.
$$
This inequality turns out to be true for  $p> 1.06$  and therefore it is true for $p>p_0$.
\end{proof}

\section{Study of the dynamic programming principle}\label{sec:dpp}
Recall the notation
\[
\Mmvp[\phi](x)=\frac{1}{D_{d,p}r^p} \fint_{ B_r} |\phi(x+y)-\phi(x)|^{p-2} (\phi(x+y)-\phi(x)) \dd y.
\]
Given an open domain $\Omega$ and $r>0$, we will in this section denote by
$$\partial \Omega_r=\{x\in \Omega^c \ :\ \textup{dist}(x,\Omega)\leq r \}$$ and $\Omega_r=\Omega\cup \partial \Omega_r$.

We want to study solutions of the (extended) boundary value problem
\begin{equation}\label{eq:DPP2}
\begin{cases}
-\Mmvp[U_r] (x)=f(x)& x\in \Omega\\
U_r(x)=G(x)& x\in \partial \Omega_r:=\{x\in \Omega^c \ :\ \textup{dist}(x,\Omega)\leq r \},
\end{cases}
\end{equation}
where $f\in C(\overline{\Omega})$ and $G\in C(\partial \Omega_r)$ (a continuous extension of $g\in C(\partial \Omega)$).  These will be our running assumptions in this section.

\subsection{Existence and uniqueness: The proof of Theorem \ref{thm:dpp} i)}

 For convenience, we will write $\Mmvpb$ instead of $\Mmvp$ when the subindex $r$ plays no role.

We first prove a comparison principle which immediately implies uniqueness and then we prove the existence.

\begin{proposition}\label{prop:comparison}
Let $p\in (1,\infty)$ and $U,V\in L^\infty(\Omega_r)$ be such that
\[
\begin{cases}
-\Mmvpb[V] (x)\geq f(x)& x\in \Omega,\\
V(x)\geq G(x)& x\in \partial \Omega_r,
\end{cases} \qquad 
\textup{and}
\qquad 
 \begin{cases}
-\Mmvpb[U] (x)\leq f(x)& x\in \Omega,\\
U(x)\leq G(x)& x\in \partial \Omega_r.
\end{cases}
\]
Then $U\leq V$ in $\Omega_r$.
\end{proposition}

\begin{proof}

Assume by contradiction that $U(x)>V(x)$ for some $x\in \Omega$. It has to be in the interior of $\Omega$ since by definition $U\leq G\leq V$ in $\partial \Omega_r$.

Let $M>0$ and $x_0\in \Omega$ be such that 
\[
M=U(x_0)-V(x_0)=\sup_{x\in \Omega} \{U(x)-V(x)\}.
\]
Define $\U=U-M$. Then $\U(x_0)=V(x_0)$, $\U\leq V$ in $\Omega$, $\U<V$ in $\partial \Omega_r$, and
\[
\begin{cases}
-\Mmvpb[\U] (x)\leq f(x)& x\in \Omega,\\
\U(x)\leq G(x)-M& x\in \partial \Omega_r.
\end{cases}
\]
By the monotonicity of $J_p$
\[
\begin{split}
 J_p(V(x_0+y)-&V(x_0))- J_p(\U(x_0+y)-\U(x_0))\\
& \geq  J_p(\U(x_0+y)-V(x_0))- J_p(\U(x_0+y)-\U(x_0))\\
& =  J_p(\U(x_0+y)-\U(x_0))- J_p(\U(x_0+y)-\U(x_0))=0.
\end{split}
\]
From the equations satisfied by $U$ and $V$ we have
\[
\begin{split}
0&\geq \Mmvpb[V](x_0)-\Mmvpb[U](x_0) =\frac{1}{D_{d,p}r^p}
\fint_{B_r} J_p(V(x_0+y)-V(x_0))- J_p(\U(x_0+y)-\U(x_0))  \dd y.
\end{split}
\]
Hence, the average of the non-negative integrand is non-positive. This means that
\[
J_p(V(x_0+y)-V(x_0))= J_p(\U(x_0+y)-\U(x_0)).
\]
By the strict monotonicity of $J_p$ this implies
\[
V(x_0+y)-V(x_0)=\U(x_0+y)-\U(x_0),
\]
that is, $V(x_0+y)=\U(x_0+y)$ for all $y\in B_r$. This means that $\U(x)=V(x)$ for all $x\in B_r(x_0)$. Repeating this process in the contact points of $\U$ and $V$ and iterating, we will eventually arrive at the conclusion that $\U(x)=V(x)$ for some $x\in \partial \Omega_r$. This contradicts the fact $\U <V$ in $\partial \Omega_r$.
\end{proof}

In order to prove the existence and to study the limit as $r \to 0$, we will first derive uniform bounds (in $r$) for the solution of \eqref{eq:DPP2}.
\begin{proposition}[$L^\infty$-bound]\label{Prop:Stability}
Let $p\in (1,\infty)$, let $R>0$ and $U_r$ be the solution of \eqref{eq:DPP2} corresponding to some $r\leq R$. Then 
\[
\|U_r\|_{\infty}\leq A
\]
with $A>0$ depending on $p, \Omega, f, g$ and $R$ (but not on $r$).
\end{proposition}
\begin{proof} Consider the function $h(x)=|x|^{\frac{p}{p-1}}$. Then $h \in C^\infty(\R^d\setminus B_1(0))$ and
\[
\Delta_p h(x)=d\left(\frac{p}{p-1}\right)^{p-1} \quad \textup{for all} \quad x\not=0.
\]
Let $C,D\in \R$ and $z\in \R^d$ to be chosen later and define
\[
\psi(x)=C-D|x-z|^{\frac{p}{p-1}} \frac{1}{d} \left(\frac{p-1}{p}\right)^{p-1} .
\]
Then 
\[
\Delta_p \psi(x) =-D \quad \textup{for all} \quad x\not=z \quad \textup{and} \quad \psi \in C^\infty(\R^d \setminus B_1(z)).
 \]
 Now take $z$ such that
 \[
 B_1(z)\cap \Omega_R = \emptyset.
 \]
 Then $\psi \in C^\infty(\Omega_R)$. By Corollary \ref{cor:consistency}, for all $x\in \Omega$ we have
 \[
- \Mmvpb[\psi](x)=-\Delta_p\psi(x)+ o_r(1)=D+ o_r(1) \geq D - \tilde{D}
 \]
 where $\tilde{D}>0$ depends only on $R$ but not on $r$. Then choose $D=\tilde{D}+\|f\|_{\infty}$ to get
 \[
 -\Mmvpb[\psi](x)\geq D- \tilde{D} = \|f\|_{\infty} \quad \textup{for all} \quad x\in \Omega.
 \]
 Finally, we choose $C$ such that $\psi(x)\geq \|G\|_\infty$ for all $x\in \partial \Omega_R$. Thus
 \[
\begin{cases}
- \Mmvpb[\psi] (x)\geq  \|f\|_{\infty} & x\in \Omega\\
\psi(x)\geq \|G\|_\infty& x\in \partial \Omega_r,
\end{cases}
 \]
 for all $r\leq R$. Then, by comparison (Proposition \ref{prop:comparison})
 \[
 U(x) \leq \psi(x) \leq \|\psi\|_{\infty}.
 \]
 Note that this bound depends on $R$ but not on $r$. A similar argument with $-\psi$ as barrier shows that $U(x) \geq - \|\psi\|_{\infty}$ and thus,
 \[
 \|U\|_\infty \leq \|\psi\|_\infty,
 \]
 which concludes the proof.
\end{proof}

The aim is now to prove the existence of a solution of \eqref{eq:DPP2}.  Before doing that, we need some auxiliary results. Define 
\[
L[\psi,\phi](x):= \frac{1}{D_{d,p}r^p} \fint_{ B_r} J_p(\phi(x+y)-\psi(x) )\dd y.
\]
\begin{lemma}\label{lemma:L}
Let $r>0$
and $\phi\in L^\infty(\Omega_r)$.
\begin{enumerate}[\rm (a)]
\item\label{lemma:L-item1} Then there exists a unique $\psi\in L^\infty(\Omega)$ such that
\[
- L[\psi,\phi](x)=f(x) \quad \textup{for all} \quad x\in \Omega.
\]
\item\label{lemma:L-item2}  Let $\psi_1$ and $\psi_2$ be such that
\[
- L[\psi_1,\phi](x) \leq  f(x) \quad \textup{and} \quad - L[\psi_2,\phi](x)\geq f(x) \quad \textup{for all} \quad x\in \Omega,
\]
then $\psi_1\leq \psi_2$ in $\Omega$.
\end{enumerate}
\end{lemma}
\begin{proof}
We start by proving the comparison principle. This will imply uniqueness. Assume that $\psi_1(x)>\psi_2(x)$ for some $x\in \Omega$. Then
\[
\begin{split}
0&= (-f(x)+f(x))  r^pD_{d,p}\\
& \geq  \fint_{ B_r} J_p(\phi(x+y)-\psi_2(x) ) - J_p(\phi(x+y)-\psi_1(x) )\dd y\\
& >  \fint_{ B_r} J_p(\phi(x+y)-\psi_2(x) ) - J_p(\phi(x+y)-\psi_2(x) )\dd y=0
\end{split}
\]
which is a contradiction. 
To prove existence we start by defining 
$$\psi_I(x)= \sup_{\Omega_r} \phi + J_p^{-1}\left( D_{d,p}r^pf(x)\right). $$
Since 
$$ \sup_{\Omega_r} \phi - \phi(x+y) \geq 0,$$
we have
\[\begin{split}
 -  L[\psi_I,\phi](x)&=- \frac{1}{D_{d,p}r^p}\fint_{B_r} J_p\left(\phi(x+y)-\sup_{\Omega_r} \phi -   J_p^{-1}(D_{d,p}r^pf(x))\right)\dd y\\
&\geq  \frac{1}{D_{d,p}r^p}\fint_{B_r} J_p\left(J_p^{-1}\left( D_{d,p}r^pf(x)\right)\right)\dd y=f(x).
\end{split}
\]
By defining 
$$\psi_I(x)=\inf_{\Omega_r} \phi + J_p^{-1}\left(  D_{d,p}r^pf(x)\right),$$
we may prove that $- L[\psi_I,\phi](x)\leq  f(x)$ in a similar manner. By continuity we can conclude that for every $x\in \Omega$, there exists a value 
$$a_{x}\in \left [\inf_{\Omega_r} \phi + J_p^{-1}\left( D_{d,p}r^pf(x)\right), \sup_{\Omega_r} \phi + J_p^{-1}\left( D_{d,p}r^pf(x)\right)\right ],$$ such that
\begin{equation}
\label{eq:ax}
-\frac{1}{D_{d,p}r^p} \fint_{B_r} J_p(\phi(x+y)-a_x)\dd y=f(x).
\end{equation}
Observe that since $J_p$ is strictly increasing, this value is unique. We may then define $\psi(x):=a_x$ for all $x\in \Omega$. Clearly, 
$$- L[\psi,\phi](x)=f(x)$$
 for all $x\in \Omega$, so the  existence is proved.

We now claim that the constructed function is continuous. It is clearly bounded so it is sufficient to prove continuity along convergent subsequences. Take $x_j\to x$ such that $a_{x_j}\to b$. By passing to the limit in the definition of $a_{x_j}$ we obtain
\[
-\frac{1}{D_{d,p}r^p} \fint_{B_r} J_p(\phi(x+y)-b)\dd y=f(x).
\]
By uniqueness of the values $a_x$ satisfying \eqref{eq:ax}, we must have $b=a_x$. Therefore, $\psi(x)=a_x$ is a continuous function.
\end{proof}

We are now ready to prove the existence.

\begin{proposition}\label{prop:existenceDPPf0}
There exists a solution $U\in L^\infty(\R^d)$ of \eqref{eq:DPP2}.
\end{proposition}
\begin{proof}
Consider $h$ to be the barrier function constructed in Proposition \ref{Prop:Stability} (denoted there by $\psi$), i.e., $h$ is such that
\[
 -\Mmvpb[h] (x)\geq  \|f\|_{\infty}\geq f(x) \quad \textup{if} \quad  x\in \Omega  \quad \textup{and} \quad h(x)\geq \|G\|_\infty\geq 0  \quad \textup{if} \quad  x\in \Omega_r.
\]
Define
\[
U_0(x)=
\begin{cases}
\displaystyle\inf_{z\in \partial \Omega_r} G(z) -h(x) &x\in \Omega,\\
G(x) &x\in \partial\Omega_r.
\end{cases}
\]
Note that if $x\in \partial \Omega_r$ then \[
U_0(x)=G(x)\geq \inf_{z\in \partial \Omega_r} G(z) \geq \inf_{z\in \partial \Omega_r} G(z) -h(x). 
\]
Thus $U_0(x)\geq \inf_{z\in \partial \Omega_r} G(z) -h(x)$ in $\Omega_r$.

We define the sequence $U_k$ as the sequence of solutions of 
\[
\begin{cases}
 -  L[U_k,U_{k-1}](x)= f(x)& x\in \Omega,\\
U_k(x)= G(x)& x\in \partial \Omega_r.
\end{cases}
\]
As long as $U_{k-1}$ is bounded, $U_k$ exists by Lemma \ref{lemma:L}\eqref{lemma:L-item1}. We now prove that $U_{k+1}(x)\geq U_{k}(x)$  in $\Omega_r$ by induction. We start by proving that \[
U_1(x) \geq  \inf_{z\in \partial \Omega_r} G(z)-h(x)= U_0(x).
\]
Assume towards a contradiction that 
\[
U_1(x) <  \inf_{z\in \partial \Omega_r} G(z)-h(x)= U_0(x)
\]
for some $x\in \Omega$. Clearly $U_0(x)=U_1(x)$ if $x\in \partial \Omega_r$. So we must have $x\in \Omega$. By the monotonicity of $J_p$ 
\[
\begin{split}
- f(x)&=   L[U_1,U_0](x)\\
&= \frac{1}{D_{d,p}r^p}\fint_{ B_r} J_p( U_0(x+y)-U_1(x) )\dd y \\
& \geq   \frac{1}{D_{d,p}r^p} \fint_{ B_r} J_p((\inf_{z\in \partial \Omega_r} G(z)-h(x+y))-U_1(x) )\dd y\\
&= \frac{1}{D_{d,p}r^p} \fint_{ B_r} J_p((\inf_{z\in \partial \Omega_r} G(z)-h(x+y)) -U_1(x) )\dd y\\
&\quad   + \frac{1}{D_{d,p}r^p} \fint_{ B_r}J_p(h(x+y)-h(x))\dd y \\
&\quad - \frac{1}{D_{d,p}r^p}  \fint_{ B_r}J_p(h(x+y)-h(x))\dd y \\
&> -\frac{1}{D_{d,p}r^p} \fint_{ B_r}J_p(h(x+y)-h(x))\dd y\\
&=  -\Mmvpb[h](x) \\
&\geq   \|f\|_\infty.
\end{split}
\]
Thus, $-f(x)>\|f\|_\infty$, which is clearly a contradiction. We conclude that
\[
U_1(x) \geq  \inf_{z\in \partial \Omega_r} G(z)-h(x)= U_0(x).
\]
Now assume that $U_k\geq  U_{k-1}$. Then
\[
\begin{split}
- f(x)&= L[U_{k+1}, U_{k}](x)\\
&= \frac{1}{D_{d,p}r^p} \fint_{ B_r} J_p( U_{k}(x+y)-U_{k+1}(x))\dd y \\
&\geq  \frac{1}{D_{d,p}r^p}  \fint_{ B_r} J_p( U_{k-1}(x+y)-U_{k+1}(x))\dd y \\
&= L[U_{k+1}, U_{k-1}](x).
\end{split}
\]
This implies
\[
\begin{cases}
- L[U_{k+1},U_{k-1}](x)\geq f(x)& x\in \Omega,\\
U_{k+1}(x)= G(x)& x\in \partial \Omega_r,
\end{cases}
\]
and
\[
\begin{cases}
- L[U_{k},U_{k-1}]= f(x)& x\in \Omega,\\
U_k= G(x)& x\in \partial \Omega_r.
\end{cases}
\]
By comparison (Lemma \ref{lemma:L}\eqref{lemma:L-item2}), $U_{k+1}\geq U_{k}$. Thus the induction is complete and the claim is proved.

We will now verify that $U_k$ is uniformly bounded from above by $\|G\|_\infty$. We argue that $U_0(x) \leq h(x) $ as follows. If $x\in \partial \Omega_r$, then
\[
U_0(x)= G(x) \leq\|G\|_\infty\leq h(x).
\]
If instead $x\in \Omega$, then
\[
U_0(x)=\inf_{z\in \partial \Omega_r} G(z) -h(x)\leq \inf_{z\in \partial \Omega_r} G(z) \leq \|G\|_\infty \leq h(x).
\]
Assume now that $U_k(x)\leq h(x)$. If $x\in \partial \Omega_r$, then
\[
 U_{k+1}(x)= G(x) \leq\|G\|_\infty\leq h(x).
\]
On the other hand, if $x\in \Omega$, then
\[
- f(x)=L[U_{k+1},U_k](x)\leq L[U_{k+1},h](x).
\]
In particular,
\[
-L[h,h](x) \geq f(x) \quad \textup{and}\quad    -L[U_{k+1},h](x) \leq f(x)
\]
which by comparison (Lemma \ref{lemma:L}\eqref{lemma:L-item2}) implies that $U_{k+1}\leq h$ and thus proves the claim.

We conclude that for every $x\in \Omega_r$, the sequence $U_k(x)$ is non-decreasing and bounded from above. We can then define the limit
\[
U(x):=\lim_{k\to \infty} U_k(x).
\]
By the monotone convergence theorem
\[
\begin{split}
-f(x)&=\lim_{k\to \infty} L[U_{k+1}, U_k](x)=  L[\lim_{k\to \infty} U_{k+1}, \lim_{k\to \infty} U_k](x)=L[U, U](x)= \Mmvpb[U](x)
\end{split}
\]
so that $U$ is a solution of \eqref{eq:DPP2}.

\end{proof}

\subsection{Convergence: The proof of Theorem \ref{thm:dpp} ii)}

The proof of the convergence is based on the numerical analysis technique introduced by Barles and Souganidis in \cite{BS91}. We partially follow the outline of \cite{dTMP18}, where this technique was adapted to homogeneous problems involving the $p$-Laplacian.

\subsubsection{The strong uniqueness property for the boundary value problem}

Our approximate problem \eqref{eq:DPP2} will produce a sequence of solutions that converges to a so-called generalized viscosity solution (see below). To complete our program we need to ensure that this solution is unique and coincides with the usual viscosity solution.
\begin{definition}[Generalized viscosity solutions of the boundary value problem] \label{def:genvisc}
Let $f$ be a continuous function in $\overline{\Omega}$ and $g$ a continuous function in $\partial \Omega$. We say that a lower (resp. upper) semicontinuous function $u$ in $\overline{\Omega}$  is a \emph{generalized viscosity supersolution} (resp. \emph{subsolution}) of \eqref{eq:BVP} in $\overline{\Omega}$ if whenever $x_0 \in \overline{\Omega}$ and $\varphi \in
  C^2( B_R(x_0))$ for some $R>0$ are such that $|\nabla \varphi(x)|\neq 0$ for $x\in B_R(x_0)\setminus\{x_0\}$, 
$$ \varphi(x_0) = u(x_0) \quad \text{and} \quad \varphi(x) \leq
u(x) \ \text{(resp. $\varphi(x)\geq u(x)$)}\quad \text{for all} \quad
x \in  B_R(x_0)\cap\overline{\Omega},$$
then we have
\begin{equation*}
\begin{split}
\lim_{\rho\to 0}\sup_{B_{\rho(x_0)}\setminus\{x_0\}}\left(-\plap\varphi (x)-f(x_0)\right)&\geq 0 \quad \textup{if} \quad x_0\in \Omega\\
 \text{(resp. } \lim_{\rho\to 0}\inf_{B_{\rho(x_0)}\setminus\{x_0\}}\left( -\plap\varphi (x)-f(x_0)\right)& \leq 0\text{)}\\
\max\left\{\lim_{\rho\to 0}\sup_{B_{\rho(x_0)}\setminus\{x_0\}}\left(-\plap\varphi (x)-f(x_0)\right), u(x_0)-g(x_0)\right\}&\geq0 \quad \textup{if} \quad x_0\in \partial\Omega\\
\Big(\text{resp. } \min\left\{\lim_{\rho\to 0}\inf_{B_{\rho(x_0)}\setminus\{x_0\}}\left( -\plap\varphi (x)-f(x_0)\right), u(x_0)-g(x_0)\right\}&\leq0\Big)
\end{split}
\end{equation*}
\end{definition} 

We need the following uniqueness results for the generalized concept of viscosity solutions.

\begin{theorem}[Strong uniqueness property]\label{thm:SUP}
Let $\Omega$ be a $C^2$ domain. If $\underline{u}$ and $\overline{u}$ are generalized viscosity subsolutions and supersolutions of \eqref{eq:BVP} respectively, then $\underline{u}\leq \overline{u}$.
\end{theorem}

The above result for standard viscosity solutions is well known (see Theorem 2.7 in \cite{equiv}). The proof of Theorem \ref{thm:SUP} follows from this fact together with the following equivalence result between the two notions of viscosity solutions. 

\begin{proposition}\label{prop:equivnotions}
Let $\Omega$ be a $C^2$ domain. Then $u$ is a viscosity subsolution (resp. supersolution) of \eqref{eq:BVP} if and only if $u$ is a generalized viscosity subsolution (resp. supersolution) of \eqref{eq:BVP}.
\end{proposition}
\begin{proof} We prove the statement for subsolutions. Clearly if $u$ is a viscosity subsolution, then it is also a generalized viscosity subsolution since
\[
\min\{-\plap \varphi (x)-f(x_0), u(x_0)-g(x_0)\}\leq u(x_0)-g(x_0)\leq0.
\]
The proof of the other implication is essentially contained in \cite{dTMP18}. We spell out the details below.

Assume $u$ is a generalized viscosity subsolution. Fix a point $x_0\in \partial \Omega$ and define, for $\veps>0$ small enough, the following function
\[
\varphi_\veps(y)=\frac{|y-x_0|^4}{\veps^4}+\frac{d(y)}{\veps^2}-\frac{d(y)^2}{2\veps^2}
\]
where $d(y):=\textup{dist}(y,\partial \Omega)$. As it is shown in the proof of Theorem 3.4 in \cite{dTMP18}, this is a suitable test function at some point $y_\veps\in  \overline{\Omega}$ as in Definition \ref{def:genvisc}. Moreover, $u(x_0)\leq u(y_\veps)$ for all $\veps>0$ small enough and
\[
y_\veps \to x_0 \quad \textup{as} \quad \veps\to0.
\]

It is standard to check, as done in step three of the proof of Theorem 3.4 in \cite{dTMP18}, that
\[
\nabla \varphi_\veps (y_\veps)\not=0 \quad \textup{for all} \quad \veps>0.
\]
which allows us to use the standard condition \eqref{eq:testfuncdef2} rather than  \eqref{eq:testfuncdef}.

By direct computations, it is also shown in step four of the proof of Theorem 3.4 in \cite{dTMP18} that 
\[
\plap \varphi_\veps(y_\veps)  \leq C_1 \veps^{2(2-p)}\left(\frac{C_2}{\veps^2}-\frac{C_3}{\veps^3}\right)
\]
for constants $C_1,C_2,C_3>0$. From here, it is standard to get that there exists a constant $C>0$ and $\veps_0>0$ such that for all $\veps<\veps_0$
\[
\plap \varphi_\veps(y_\veps) <-C \frac{1}{\veps^{2p-1}}<-\|f\|_\infty\leq-  f(y_\veps).
\]
Thus,
\[
-\plap \varphi_\veps(y_\veps)- f(y_\veps)>0.
\]
This implies that $y_\veps\in \partial \Omega$. Indeed, if $y_\veps\in \Omega$ then by definition of generalized viscosity subsolution we have $-\plap \varphi_\veps(y_\veps)- f(y_\veps)\leq0$. Since $y_\veps\in \partial \Omega$, then we have by definition that 
\[
\min\{-\plap \varphi (y_\veps)-f(y_\veps), u(y_\veps)-g(y_\veps)\}\leq0
\]
which implies that $u(y_\veps)-g(y_\veps)\leq0$. Finally, using the fact that $u(x_0)\leq u(y_\veps) $ and taking the limit as
$\veps\to0$, we obtain $u(x_0)-g(x_0)\leq0$, since $g$ is continuous.  This shows $u$ is a viscosity subsolution. 
\end{proof}
Note that the restriction of having a $C^2$ domain in the proposition above comes from the fact that we need the distance function to be $C^2$ close to the boundary.
\subsubsection{Monotonicity and consistency of the approximation}
For convenience we define
\[
S(r,x,\phi(x),\phi):=\begin{cases}
\displaystyle - \frac{1}{D_{d,p}r^p} \fint_{ B_r} J_p(\phi(x+y)-\phi(x))\dd y - f(x)  &x\in \Omega,\\
\phi(x)-G(x) &x\in \partial\Omega_r.
\end{cases}
\]
Note that \eqref{eq:DPP2} can then be equivalently formulated as
\[
S(r,x,U_r(x),U_r)=0 \quad x\in \Omega_r.
\]
We have the following properties for $S$:
\begin{lemma}\label{lem:propscheme}
\begin{enumerate}[\rm(a)]
\item\label{lem:propscheme-item-a} \textup{(Monotonicity)} Let $t\in\R$ and $\psi\geq\phi$. Then
\[
S(r,x,t,\psi)\leq S(r,x,t,\phi)
\]
\item\label{lem:propscheme-item-b}  \textup{(Consistency)}  For all $x\in \overline{\Omega}$ and  $\phi\in C^2( B_R(x))$ for some $R>0$  such that $|\nabla \phi(x)|\neq 0$ we have that
\[
\limsup_{r\to0,\ z\to x,\ \xi \to 0} S(r, z, \phi(z)+\xi+\eta_{r}, \phi+\xi)\leq \begin{cases}
\displaystyle-\Delta_p\phi(x)-f(x)& \text{if } x\in\Omega\\
\max \left\{\displaystyle-\Delta_p \phi(x)-f(x), \phi(x)-g(x)\right\}& \text{if } x\in \partial{\Omega},
\end{cases}
\]
\noindent and 
\[
\liminf_{r\to0,\ z\to x,\ \xi \to 0} S(r, z, \phi(z)+\xi-\eta_{r}, \phi+\xi)\geq
\begin{cases}
\displaystyle-\Delta_p\phi(x)-f(x)& \text{if } x\in\Omega\\
\min\left\{\displaystyle-\Delta_p\phi(x)-f(x), \phi(x)-g(x)\right\}& \text{if } x\in \partial{\Omega},
\end{cases}
\]
where $\eta_{r}\ge 0,\ \eta_{r}/r^{p}\to 0$ as $r\to 0$.
\end{enumerate}
\end{lemma}

\begin{proof}
Note that \eqref{lem:propscheme-item-a} is trivial. For \eqref{lem:propscheme-item-b}, let first $x\in \Omega$. Recall that $\xi\mapsto J_p(\xi)$ is a H\"older continuous function with exponent $\delta=\min\{p-1,1\}>0$. Then, using basic properties of the $\limsup$, consistency for smooth functions of Theorem \ref{thm:MVF} and the continuity of $-\Delta_p\phi$, we get

\begin{equation*}
\begin{split}
 \limsup_{r\to0,\, z\to x,\, \xi \to 0} S(r, z,  \phi(z)+&\xi+\eta_{r}, \phi+\xi)\\
 & =  \limsup_{r\to0,\, \Omega\ni z\to x,\, \xi \to 0} S(r, z,  \phi(z)+\xi+\eta_{r}, \phi+\xi)\\
& =\limsup_{r\to0, z\to x} \left( \frac{1}{D_{d,p}r^p} \fint_{ B_r} J_p(\phi(z)-\phi(z+y) + \eta_r)\dd y - f(z)\right)\\
& \leq  \limsup_{r\to0, z\to x} \left( \frac{1}{D_{d,p}r^p}  \fint_{ B_r} J_p(\phi(z)-\phi(z+y))\dd y - f(z)+ C \left(\frac{\eta_r}{r^p}\right)^{ \delta}  \right)\\
&=  \limsup_{r\to0, z\to x} \left(-\Delta_p \phi(z)-f(z)+ o_r(1)+ C \left(\frac{\eta_r}{r^p}\right)^{ \delta}  \right)\\
& =  \limsup_{z\to x}\left( -\Delta_p \phi(z)-f(z)\right)+  \limsup_{r\to0} \left( o_r(1)+ C \left(\frac{\eta_r}{r^p}\right)^{ \delta} \right)\\
& = -\Delta_p \phi(x)-f(x).
\end{split}
\end{equation*}
If $x\in \partial \Omega$, we simply note that
\begin{equation*}
\begin{split}
\limsup_{r\to0,\, z\to x,\, \xi \to 0} &S(r, z, \phi(z)+\xi+\eta_{r}, \phi+\xi)\\
&\quad = \max \bigg\{ \limsup_{r\to0,\, \Omega\ni z\to x,\, \xi \to 0} \frac{1}{D_{d,p}r^p}\fint_{ B_r} J_p(\phi(z)-\phi(z+y) + \eta_r)\dd y - f(z) ,  \\
&\hspace{6.5cm},\limsup_{r\to 0,  \Omega^c\ni z\to x,\, \xi \to 0} (\phi(z)+\eta_{r}-G(z)+\xi)\bigg\}\\
&\quad \leq  \max \left\{-\Delta_p \phi(x)-f(x), \phi(x)-g(x)\right\}.
\end{split}
\end{equation*}
A similar argument works for the $\liminf$.
\end{proof}

\subsubsection{Proof of the convergence} The only thing left to show is the convergence stated in Theorem \ref{thm:dpp}.  Once we have proved  monotonicity and consistency as stated in Lemma \ref{lem:propscheme}, the proof follows as explained in Section 4.3 of \cite{dTMP18}. 

\begin{proof}[Proof of Theorem \ref{thm:dpp} ii)] Define 
\begin{equation*}
\overline{u}(x)=\limsup_{r\to 0,\, y\to x} U_r(y), \qquad \underline{u}(x)=\liminf_{r\to 0,\, y\to x} U_r(y)
\end{equation*}
By definition $\underline{u}\leq\overline{u}$ in $\overline{\Omega}$. If we show that $\overline{u}$ (resp. $\underline{u}$) is a generalized viscosity subsolution (resp. supersolution) of \eqref{eq:BVP}, the strong uniqueness property of Theorem \ref{thm:SUP} ensures that $\overline{u}\leq \underline{u}$. Thus, $u:=\overline{u}=\underline{u}$ is a generalized viscosity solution of \eqref{eq:BVP} and  $U_r\to u$ uniformly in $\overline{\Omega}$ (see \cite{BS91}). Proposition \ref{prop:equivnotions} ensures that $u$ is a viscosity solution \eqref{eq:BVP}. 

We need to show that $\overline{u}$ is a generalized viscosity subsolution. First note that  $\overline{u}$ is an upper semicontinuous function by definition, and it is also bounded since $U_r$ is uniformly bounded by Proposition \ref{Prop:Stability}. 
Take $x_0\in \overline{\Omega}$ and $\varphi\in C^2(B_R(x_0))$ such that $\overline{u}(x_0)=\varphi(x_0)$, $\overline{u}(x)<\varphi(x)$ if $x\not=x_0$. We separate the proof into different cases depending on the value of the gradient of $\varphi$ at $x_0$.
 
 \textit{Case 1:} Let $\nabla \varphi(x_0)\not=0$. In this case, we can consider the standard condition \eqref{eq:testfuncdef2} rather than  \eqref{eq:testfuncdef}. Then, for all $x\in \overline{\Omega}\cap B_R(x_0)\setminus\{x_0\}$, we have that
\begin{equation}\label{x0localmax}
 \overline{u}(x)-\varphi(x)<0=  \overline{u}(x_0)-\varphi(x_0).
 \end{equation}
 We claim that we can find a sequence $(r_n,y_n)\to(0,x_0)$ as $n\to \infty$ such that 
\begin{equation}\label{eq:famseq}
 U_{r_n}(x)-\varphi(x) \leq U_{r_n}(y_n)-\varphi(y_n)+ e^{-1/r_n}   \quad \textup{for all} \quad x\in \overline{\Omega}\cap B_R(x_0).
 \end{equation}
 To show this, we consider a sequence $(r_j,x_j)\to(0,x_0)$ as $j\to \infty$ such that $U_{r_j}(x_j)\to \overline{u}(x_0)$, which exists by definition of $\overline{u}$. For each $j$, there exists $y_j$ such that
 \begin{equation}\label{eq:exp}
  U_{r_j}(y_j)-\varphi(y_j)+ e^{-1/r_j}  \geq \sup_{\overline{B}_r(x_0)}\{U_{r_j}-\varphi\}.
  \end{equation}
Now extract a subsequence $(r_n,x_n,y_n)\to (0, x_0, \hat{y})$ as $n\to \infty$ for some $\hat{y}\in \overline{\Omega}$. Then,
 \begin{equation*}
 \begin{split}
 0&= \overline{u}(x_0)-\varphi(x_0)\\
 &= \lim_{n\to\infty} \left\{U_{r_n}(x_n)-\varphi(x_n)\right\}\\
 &\leq \limsup_{n\to\infty}\left\{U_{r_n}(y_n)-\varphi(y_n)+ e^{-1/r_n}\right\}\\
 &\leq \limsup_{r\to0, y \to \hat{y}}\left\{U_{r}(y)-\varphi(y)+ e^{-1/r}\right\}\\
 &= \overline{u}(\hat{y})- \varphi(\hat{y}),
 \end{split}
 \end{equation*}
 where we in the third inequality have used \eqref{eq:exp}. This together with \eqref{x0localmax} implies that $\hat{y}=x_0$ and thus finishes proof of the claim. 
 
 Choose now $\xi_n:=U_{r_n}(y_n)-\varphi(y_n)$. We have from \eqref{eq:famseq} that,
 \[
 U_{r_n}(x)\leq \varphi(x) + \xi_n + e^{-1/r_n} \quad \textup{for all} \quad x\in \overline{\Omega}\cap B_R(x_0).
 \]
 From the monotonicity given in Lemma \ref{lem:propscheme}\eqref{lem:propscheme-item-a} we thus get,
 \begin{equation*}
 \begin{split} 
 0&=S(r_n,y_n, U_{r_n}(y_n),U_{r_n})\\
 &=S(r_n,y_n, \varphi(y_n)+\xi_n,U_{r_n})\\
 &\geq S(r_n,y_n, \varphi(y_n)+\xi_n,\varphi + \xi_n + e^{-1/r_n} )\\
 &=S(r_n,y_n, \varphi(y_n)+\xi_n- e^{-1/r_n} ,\varphi + \xi_n ).
 \end{split}
 \end{equation*}
 Note that  $e^{-1/r}=o(r^p)$. By the consistency, Lemma \ref{lem:propscheme}\eqref{lem:propscheme-item-b}, we have 
  \begin{equation*}
 \begin{split} 
 0 & \geq \liminf_{r_n\to0,\, y_n\to x_0,\, \xi_n\to0}S(r_n,y_n, \varphi(y_n)+\xi_n- e^{-1/r_n} ,\varphi + \xi_n )\\
 & \geq \liminf_{r\to0,\, y\to x_0,\, \xi \to0}S(r,y, \varphi(y)+\xi- e^{-1/r} ,\varphi + \xi )\\
 & \geq \left\{\begin{array}{cccl}
\displaystyle-\Delta_p\varphi(x_0)-f(x_0)& \text{ if } &x_0\in\Omega,\\
\displaystyle\min\{-\Delta_p\varphi(x_0)-f(x_0), \overline{u}(x_0)-g(x_0)\}& \text{ if } &x_0\in \partial{\Omega},
\end{array}\right.
 \end{split}
 \end{equation*}
 which are the required inequalities in this case. 
 
\textit{Case 2:} Let $\nabla \varphi(x_0)=0$ and assume that $\overline u$ happens to be constant in some ball $B_\rho(x_0)$ for $\rho>0$ small enough.   
Then the function $\phi(x)=\overline u(x_0)+|x-x_0|^{\frac{p}{p-1}+1}$ touches $\bar u$ from above at $x_0$  and it is a suitable test function since $\nabla\phi(x)\not=0$ if $x\not=x_0$.  Arguing as before we get
\[
0 \geq \liminf_{r\to0,\, y\to x_0,\, \xi \to0}S(r,y, \phi(y)+\xi- e^{-1/r} ,\phi + \xi ).
\]
Assume for simplicity that $x_0\in \Omega$. The case $x_0\in \partial \Omega$ follows similarly. Following the proof of Lemma \ref{lem:propscheme}\eqref{lem:propscheme-item-b}, the above inequality implies
\begin{equation}\label{eq:testspecial}
0\geq  \liminf_{r\to0,\, y\to x_0}  \frac{1}{D_{d,p}r^p}  \fint_{ B_r} J_p(\phi(y)-\phi(y+z))\dd z - f(x_0)
\end{equation}
Now by Lemma \ref{lem:xbeta}, 
\[
 \liminf_{r\to0,\, y\to x_0}  \frac{1}{D_{d,p}r^p}  \fint_{ B_r} J_p(\phi(y)-\phi(y+z))\dd z=0.
\]
On the other hand, since $\overline u$ is constant around $x_0$, we also have $\Delta_p \overline u(x_0)=0$. All together we get from \eqref{eq:testspecial} that 
\[
-\Delta_p \overline u(x_0)=0\leq f(x_0),
\]
that is, $\overline u(x_0)$ is a classical subsolution and then also a viscosity subsolution.
 
 \textit{Case 3:}  Let $\nabla \varphi(x_0)=0$ and assume that $\overline u$ is not constant in any ball $B_\rho(x_0)$.  Then we may argue as in the proof of Proposition 2.4 in \cite{AR18} to prove that there is a sequence $y_k\to 0$ such that the functions $\varphi_k(x)=\varphi(x+y_k)$ touches $\overline u$ from above at points $x_k$ and $\nabla \varphi_k(x_k)\neq 0$ for all $k$. As in Case 1, we get
 $$
 0\geq  \left\{\begin{array}{cccl}
\displaystyle-\Delta_p\varphi(x_k)-f(x_k)& \text{ if } &x_k\in\Omega,\\
\displaystyle\min\{-\Delta_p\varphi(x_k)-f(x_k), \overline{u}(x_k)-g(x_k)\}& \text{ if } &x_k\in \partial{\Omega},
\end{array}\right.
 $$
By passing $k\to\infty$ we obtain the desired inequalities also in this case.

The steps above together show that $\overline{u}$ is a viscosity subsolution and finishes the proof.
 \end{proof}

\appendix

\section{Auxiliary inequalities}\label{sec:aux}
 We need some technical results.
\begin{lemma}{\label{lem:pineq1}}
Let $p\geq 2$ and $\e\in [0,p-2)$. Then 
$$
\Big||a+b|^{p-2}(a+b)-|a|^{p-2}a-(p-1)|a|^{p-2}b\Big|\leq C\max( |a|,|a+b|)^{p-2-\e}|b|^{1+\e}, 
$$
where $C=C(p,\e)$.
\end{lemma}
\begin{proof}
It follows from the Taylor expansion of the function $J_p(t)=|t|^{p-2}t.$
\end{proof} The following inequality is Lemma 3.4 in \cite{KKL}.
\begin{lemma}{\label{lem:pineq3}} Let $p\in (1,2)$. Then 
$$
\Big||a+b|^{p-2}(a+b)-|a|^{p-2}a\Big|\leq C\left(|a|+|b|\right)^{p-2}|b|.
$$
Here $C$ only depends on $p$.
\end{lemma}
 We also need the following  lemma.
\begin{lemma}\label{lem:stupid} Let $s\in (0,1)$ and let $L$ be a quadratic form  in $\R^d$  such that 
\begin{equation}\label{eq:QFest}
|L(\omega,\omega)|< \frac{1}{d^2+1},
\end{equation}
for all $ |\omega| =  1.$ Then 
$$
\int_{S^d}|{e}_1\cdot \omega+L(\omega,\omega)|^{-s} \dd\omega \leq C, 
$$
where $C$ only depends on $s$ and $d$.

\end{lemma}


\begin{proof}
We use spherical coordinates, $( \omega_1,\omega_i)=(\cos \theta_1, \sin \theta_1 g_i(\theta_i,\ldots, \theta_{d-1}))$. By symmetry it is enough to consider the range $\theta_1:0\to  \pi/2$. Now, if $\theta\in (0,\pi/4)$ then 
$$
\cos \theta_1>1/\sqrt{2}
$$
so that 
\begin{equation}\label{eq:pi4}
e_1\cdot \omega+L(\omega,\omega)\geq 1/\sqrt{2}-1/2>0. 
\end{equation}
Hence the integral over that interval is bounded by some constant.
\medskip

On the other part of the interval we introduce, for fixed $(\theta_2,\ldots, \theta_{d-1})$, the function
$$f(\theta_1)={e}_1\cdot \omega+L(\omega,\omega).$$
We note that, due to the bounds on $L$, 
\[
\begin{split}
{e}_1\cdot \omega+L(\omega,\omega) &=\cos \theta_1+\lambda_1\cos^2 \theta_1+\sin^2 \theta_1\sum_{i=2}^{d} \lambda_i g^2_i \\
&=\cos \theta_1+\left(\lambda_1-\sum_{i=2}^{d} \lambda_i g^2_i\right)\cos^2\theta_1+\sum_{i=2}^{d} \lambda_i g^2_i\\
&= \cos \theta_1+A\cos^2\theta_1 +B, 
\end{split}
\]
where $\lambda_i$ denote the eigenvalues of $L$ and where $A$ and $B$ are functions of $\theta_2,\ldots,\theta_{d-1}$, with $|A|,|B|<1/2$. 
We have 
$$
f'(\theta_1)=-\sin \theta_1-2A\sin \theta_1\cos \theta_1 = -\sin\theta_1(1+2A\cos \theta_1),
$$
where $|2A\cos \theta_1|\leq 1/\sqrt{2}$ and $\sin \theta_1>1/\sqrt{2}$ when $\theta_1\in (\pi/4,\pi/2)$. Therefore, 
$f'(\theta_1)<-C<0$ when $\theta_1\in (\pi/4,\pi/2)$. For this reason, $f$ may change sign at most once in the interval $(\pi/4,\pi/2)$. Suppose it happens at $\theta_0$. We may then write
\[
\begin{split}
\int_{\pi/4}^{\theta_0} |f(t)|^{-s} dt&=\int_{\pi/4}^{\theta_0} f(t)^{-s} dt \\
&\leq C\int_{\pi/4}^{\theta_0} -f'(t)f(t)^{-s} dt \\
&=C(s)\left[-f(t)^{-s+1}\right]_{\pi/4}^{\theta_0}\\
&\leq C(s)f(\pi/4)^{-s+1}\\
&\leq C(s).
\end{split}
\]
Similarly
$$
\int_{\theta_0}^{\pi/2} |f(t)|^{-s} dt =\int_{\theta_0}^{\pi/2} (-f(t))^{-s} dt \leq C(s)f(\pi/2)^{-s+1}\leq C(s).
$$
Integration over the other angles $\theta_2,\ldots,\theta_{n-1}$ yields the desired bound.
\end{proof}

\begin{lemma}\label{lem:xbeta} Assume $p\in (1,2)$ and let $\phi(x)=|x|^\beta$ with $\beta>p/(p-1)$. Then 
\[
\lim_{r\to 0,x\to 0}
\frac{1}{r^p} \fint_{\partial B_r} |\phi(x+y)-\phi(x)|^{p-2} (\phi(x+y)-\phi(x)) \dd \sigma(y)=0. 
\]
\end{lemma}
\begin{proof} 
If $x=0$, we have $|\phi(x+y)-\phi(x)|=|x|^\beta= o(|y|^{\frac{p}{p-1}}) $. Then
\[
\left|\fint_{B_r} J_p(\phi(x+y)-\phi(x)) \dd y\right| = \fint_{B_r} |J_p(o(|y|^{\frac{p}{p-1}}))|dz = o(r^p).
\]
Assume now that $x\not=0$ so that $\nabla \phi(x)\not=0$. We can use the symmetry of $J_p(y\cdot \nabla \phi(x))$ and Lemma \ref{lem:pineq3} to conclude that
\begin{equation*}
\begin{split}
\left|\fint_{B_r} J_p(\phi(x+y)-\phi(x)) \dd y\right| &= \left|\fint_{B_r} J_p( y\cdot \nabla\phi (x) + \frac{1}{2}y^T D^2 \phi(\xi)y ) \dd y\right| \\
&\leq \fint_{B_r}  \Big||y\cdot \nabla\phi (x)| + |y|^2 \sup_{\xi\in B_{r}(x)}| D^2 \phi(\xi)| \Big|^{p-2} |y|^2\sup_{\xi\in B_{r}(x)}| D^2 \phi(\xi)|  \dd y.
\end{split}
\end{equation*}

We may assume that $x$ lies in the $e_1$-direction and write $\nabla \phi (x) =\beta|x|^{\beta-1}\hat e_1:=c e_1$  for some $c>0$.  Then
\[
\begin{split}
\frac{1}{r^p}\fint_{ B_r} \Big|y\cdot \nabla\phi (x) &+ |y|^2\sup_{\xi\in B_{r}(x)}| D^2 \phi(\xi)| \Big|^{p-2} |y|^2\sup_{\xi\in B_{r}(x)}| D^2 \phi(\xi)| \dd y\\
& =  \frac{1}{r^p}\fint_{ B_r}c^{p-2}r^{p} \Big||\hat y\cdot e_1| + c^{-1}r\sup_{\xi\in B_{r}(x)}| D^2 \phi(\xi)| \Big|^{p-2} \sup_{\xi\in B_{r}(x)}| D^2 \phi(\xi)| \dd y\\
&\leq c^{p-2}\fint_{ B_r} \Big||\hat y\cdot e_1| + c^{-1}C(|x|+r)^{\beta-2}r \Big|^{p-2} C(|x|+r)^{\beta-2} \dd y\\
&\leq c^{p-2}\Big|1+ c^{-1}C(|x|+r)^{\beta-2}r \Big|^{p-2} C(|x|+r)^{\beta-2}  \\
&  = \Big|c+ C(|x|+r)^{\beta-2}r \Big|^{p-2} C(|x|+r)^{\beta-2} \\
&  \leq  \Big||x|^{\beta-1}+C(|x|+r)^{\beta-2}r \Big|^{p-2} C(|x|+r)^{\beta-2} 
\end{split}
\]
 where the third and the forth estimates are due to the fact that $|D^2 \phi (\xi)|\leq C|\xi|^{\beta-2}\leq C(|x|+r)^{\beta-2}$ if $\xi \in B_{r}(x)$ and Lemma 3.5 in \cite{KKL}.  
If $|x|\leq r$ then we obtain the estimate
\begin{equation*}
\begin{split}
  \Big||x|^{\beta-1}+C(|x|+r)^{\beta-2}r \Big|^{p-2} C(|x|+r)^{\beta-2} &\lesssim  |(|x|+r)^{(\beta-2)}r |^{p-2} (|x|+r)^{\beta-2}\\
 &\leq (|x|+r)^{(\beta-2)(p-1)} r^{p-2} \\
 &\lesssim r^{(\beta-2)(p-1)+p-2} = o_r(1).
\end{split}
\end{equation*}
 If instead $r\leq |x|$ we obtain 
\begin{equation*}
\begin{split}
  \Big||x|^{\beta-1}+C(|x|+r)^{\beta-2}r \Big|^{p-2} C(|x|+r)^{\beta-2} &\lesssim   |x|^{(\beta-1)(p-2)} (|x|+r)^{\beta-2}\\
 &\lesssim  |x|^{(\beta-1)(p-2)+\beta-2}  = o_{|x|}(1).
\end{split}
\end{equation*}

\end{proof}

\section*{Acknowledgements}
F. del Teso was partially supported by PGC2018-094522-B-I00 from the MICINN of the Spanish Government. E. Lindgren is supported by the Swedish Research Council, grant no. 2017-03736. Part of this work was carried out when the first author was visiting Uppsala University. The math department and its facilities are kindly acknowledged.

\bibliographystyle{abbrv}


\bibliography{Bibliography} 

\begin{thebibliography}{10}

\bibitem{AMRT10}
F.~Andreu-Vaillo, J.~M. Maz\'{o}n, J.~D. Rossi, and J.~J. Toledo-Melero.
\newblock {\em Nonlocal diffusion problems}, volume 165 of {\em Mathematical
  Surveys and Monographs}.
\newblock American Mathematical Society, Providence, RI; Real Sociedad
  Matem\'{a}tica Espa\~{n}ola, Madrid, 2010.

\bibitem{AHP17}
{\'A}.~Arroyo, J.~Heino, and M.~Parviainen.
\newblock Tug-of-war games with varying probabilities and the normalized
  {$p(x)$}-{L}aplacian.
\newblock {\em Commun. Pure Appl. Anal.}, 16(3):915--944, 2017.

\bibitem{AL16}
A.~Arroyo and J.~G. Llorente.
\newblock On the asymptotic mean value property for planar {$p$}-harmonic
  functions.
\newblock {\em Proc. Amer. Math. Soc.}, 144(9):3859--3868, 2016.

\bibitem{AR18}
A.~Attouchi and E.~Ruosteenoja.
\newblock Remarks on regularity for {$p$}-{L}aplacian type equations in
  non-divergence form.
\newblock {\em J. Differential Equations}, 265(5):1922--1961, 2018.

\bibitem{BS91}
G.~Barles and P.~E. Souganidis.
\newblock Convergence of approximation schemes for fully nonlinear second order
  equations.
\newblock {\em Asymptotic Anal.}, 4(3):271--283, 1991.

\bibitem{BS18}
C.~Bucur and M.~Squassina.
\newblock An asymptotic expansion for the fractional $p$-laplacian and gradient
  dependent nonlocal operators.
\newblock {\em Preprint: arXiv:2001.09892}, 2020.

\bibitem{dTMP18}
F.~del Teso, J.~J. Manfredi, and M.~Parviainen.
\newblock Convergence of dynamic programming principles for the $p$-laplacian.
\newblock {\em To appear in Adv. Calc. Var. Preprint: arXiv:1808.10154}, 2018.

\bibitem{GS12}
T.~Giorgi and R.~Smits.
\newblock Mean value property for {$p$}-harmonic functions.
\newblock {\em Proc. Amer. Math. Soc.}, 140(7):2453--2463, 2012.

\bibitem{HAR16}
H.~Hartikainen.
\newblock A dynamic programming principle with continuous solutions related to
  the {$p$}-{L}aplacian, {$1 < p < \infty$}.
\newblock {\em Differential Integral Equations}, 29(5-6):583--600, 2016.

\bibitem{IMW17}
M.~Ishiwata, R.~Magnanini, and H.~Wadade.
\newblock A natural approach to the asymptotic mean value property for the
  {$p$}-{L}aplacian.
\newblock {\em Calc. Var. Partial Differential Equations}, 56(4):Art. 97, 22,
  2017.

\bibitem{IM}
T.~Iwaniec and J.~J. Manfredi.
\newblock Regularity of {$p$}-harmonic functions on the plane.
\newblock {\em Rev. Mat. Iberoamericana}, 5(1-2):1--19, 1989.

\bibitem{equivalence2}
V.~Julin and P.~Juutinen.
\newblock A new proof for the equivalence of weak and viscosity solutions for
  the {$p$}-{L}aplace equation.
\newblock {\em Comm. Partial Differential Equations}, 37(5):934--946, 2012.

\bibitem{equiv}
P.~Juutinen, P.~Lindqvist, and J.~J. Manfredi.
\newblock On the equivalence of viscosity solutions and weak solutions for a
  quasi-linear equation.
\newblock {\em SIAM J. Math. Anal.}, 33(3):699--717, 2001.

\bibitem{KMP12}
B.~Kawohl, J.~Manfredi, and M.~Parviainen.
\newblock Solutions of nonlinear {PDE}s in the sense of averages.
\newblock {\em J. Math. Pures Appl. (9)}, 97(2):173--188, 2012.

\bibitem{KKL}
J.~Korvenp\"{a}\"{a}, T.~Kuusi, and E.~Lindgren.
\newblock Equivalence of solutions to fractional {$p$}-{L}aplace type
  equations.
\newblock {\em J. Math. Pures Appl. (9)}, 132:1--26, 2019.

\bibitem{Le20}
M.~Lewicka.
\newblock Random tug of war games for the $p$-laplacian: $1<p<+\infty$.
\newblock {\em To appear in Indiana Univ. Math. J. Preprint:
  arXiv:1810.03413v}, 2020.

\bibitem{LeMaRi19}
M.~Lewicka, J.~Manfredi, and D.~Ricciotti.
\newblock Random walks and random tug of war in the heisenberg group.
\newblock {\em Mathematische Annalen}, 2019.

\bibitem{LeMa17}
M.~Lewicka and J.~J. Manfredi.
\newblock The obstacle problem for the {$p$}-laplacian via optimal stopping of
  tug-of-war games.
\newblock {\em Probab. Theory Related Fields}, 167(1-2):349--378, 2017.

\bibitem{LM16}
P.~Lindqvist and J.~Manfredi.
\newblock On the mean value property for the {$p$}-{L}aplace equation in the
  plane.
\newblock {\em Proc. Amer. Math. Soc.}, 144(1):143--149, 2016.

\bibitem{MRP}
J.~J. Manfredi, M.~Parviainen, and J.~D. Rossi.
\newblock An asymptotic mean value characterization for {$p$}-harmonic
  functions.
\newblock {\em Proc. Amer. Math. Soc.}, 138(3):881--889, 2010.

\bibitem{MO}
M.~Medina and P.~Ochoa.
\newblock On viscosity and weak solutions for non-homogeneous {$p$}-{L}aplace
  equations.
\newblock {\em Adv. Nonlinear Anal.}, 8(1):468--481, 2019.

\bibitem{ObermanpLap}
A.~M. Oberman.
\newblock Finite difference methods for the {I}nfinity {L}aplace and
  $p$-{L}aplace equations.
\newblock {\em Journal of Computational and Applied Mathematics}, 2012.

\end{thebibliography}

\end{document}